\documentclass{amsart}
\usepackage{amsfonts,graphicx,rotating,ctable}
\usepackage{amssymb}
\usepackage{hyperref}

\makeatletter
\@namedef{subjclassname@2020}{\textup{2020} Mathematics Subject Classification}
\makeatother

\newcommand{\height}{\operatorname{ht}}

\newcommand{\im}{\operatorname{im}}
\newcommand{\zcl}{\operatorname{zcl}}

\newtheorem{theorem}{Theorem}[section]
\newtheorem{lemma}[theorem]{Lemma}
\newtheorem{corollary}[theorem]{Corollary}
\newtheorem{proposition}[theorem]{Proposition}
\newtheorem{conjecture}[theorem]{Conjecture}

\theoremstyle{definition}

\theoremstyle{remark}
\newtheorem{remark}[theorem]{Remark}

\numberwithin{equation}{section}


\begin{document}

\title[Topological complexity of oriented Grassmann manifolds]{Topological complexity of oriented Grassmann manifolds}

\author{Uro\v s A.\ Colovi\'c}
\address{University of Belgrade,
  Faculty of mathematics,
  Studentski trg 16,
  Belgrade,
  Serbia}
\email{mm21033@alas.matf.bg.ac.rs}

\author{Branislav I.\ Prvulovi\'c}
\address{University of Belgrade,
  Faculty of mathematics,
  Studentski trg 16,
  Belgrade,
  Serbia}
\email{branislav.prvulovic@matf.bg.ac.rs}
\thanks{The second author was partially supported by the Ministry of Science, Technological Development and Innovations of the Republic of Serbia [contract no.\ 451-03-47/2023-01/200104].}

\author{Marko Radovanovi\'c}
\address{University of Belgrade,
  Faculty of mathematics,
  Studentski trg 16,
  Belgrade,
  Serbia}
\email{marko.radovanovic@matf.bg.ac.rs}
\thanks{The third author was partially supported by the Ministry of Science, Technological Development and Innovations of the Republic of Serbia [contract no.\ 451-03-47/2023-01/200104].}



\begin{abstract}
We study the $\mathbb Z_2$-zero-divisor cup-length, denoted by
$\zcl_{\mathbb Z_2}(\widetilde G_{n,3})$, of the Grassmann manifolds
$\widetilde G_{n,3}$ of oriented $3$-dimensional vector subspaces in $\mathbb R^n$.
Some lower and upper bounds for this invariant are obtained for all integers $n\ge6$.
For infinitely many of them the exact value of $\zcl_{\mathbb Z_2}(\widetilde G_{n,3})$ is
computed, and in the rest of the cases these bounds differ by 1.
We thus establish lower bounds for the topological complexity of
Grassmannians $\widetilde G_{n,3}$.
\end{abstract}

\maketitle

\section{Introduction}

The topological complexity of a path connected space $X$, which we denote by $\mathrm{TC}(X)$, was defined in \cite{Farber} as the Schwarz genus of the fibration $\pi:X^I\rightarrow X\times X$ ($X^I$ is the space of free paths in $X$) given by $\pi(\omega)=(\omega(0),\omega(1))$, i.e., as the minimal integer $m$ such that there exists an open cover $\{U_1,\ldots,U_m\}$ of $X\times X$ with the property that $\pi$ has a section $s_i:U_i\rightarrow X^I$ over each $U_i$, $1\le i\le m$.
Computing the exact value of topological complexity is in general very difficult problem.
Already for real projective spaces computing topological complexity is equivalent to finding
immersion dimension (see \cite{FarberTabachnikovYuzvinsky}), which is a long-standing open
problem in topology.

In \cite{Farber} a lower bound (in terms of cohomology) for $\mathrm{TC}(X)$ was detected -- the zero-divisor cup-length of $X$.
For a commutative ring $R$ and a graded $R$-algebra $A$, let $A\otimes A\stackrel{.}\longrightarrow A$ be the multiplication map. The \em zero-divisor cup-length \em of $A$, denoted by $\zcl(A)$, is defined as the supremum of the set of all integers $d$ such that there exist elements of positive degree $z_1,z_2,\ldots,z_d\in\ker\big(A\otimes A\stackrel{.}\longrightarrow A\big)$ with the property that the product $z_1z_2\cdots z_d$ is nontrivial in $A\otimes A$. For a space $X$ we define the \em $R$-zero-divisor cup-length \em of $X$, $\zcl_R(X):=\zcl\big(H^*(X;R)\big)$. The statement of \cite[Theorem 7]{Farber} is that for any field $R$ one has
\[\mathrm{TC}(X)\ge1+\zcl_R(X).\]

For integers $n$ and $k$ such that $n\ge2k\ge2$, let $G_{n,k}$ be the Grassmann manifold of $k$-dimensional subspaces in $\mathbb R^n$, and $\widetilde G_{n,k}$ the Grassmann manifold of \em oriented \em $k$-dimensional subspaces in $\mathbb R^n$. In the recent works of Pave\v si\'c \cite{Pavesic} and Radovanovi\'c \cite{Radovanovic} the $\mathbb Z_2$-zero-divisor cup-length of $G_{n,k}$ was studied for some values of $n$ and $k$. When it comes to "oriented" Grassmannians $\widetilde G_{n,k}$, in \cite{Ramani} Ramani calculated $\zcl_{\mathbb Q}(\widetilde G_{n,k})$ for all $n$ and $k$. She also computes $\zcl_{\mathbb Z_2}(\widetilde G_{n,3})$ for $6\le n\le11$, and notes that in these cases (and most likely in all others) $\zcl_{\mathbb Z_2}(\widetilde G_{n,3})$ is a better lower bound for $\mathrm{TC}(\widetilde G_{n,3})$ than $\zcl_{\mathbb Q}(\widetilde G_{n,3})$.

In this paper we obtain lower bounds for $\zcl_{\mathbb Z_2}(\widetilde G_{n,3})$ for all
$n\ge6$. We prove that if either $2^t-1\le n<2^t+2^{t-1}/3+1$ or $2^t+2^{t-1}+2^{t-2}+1\le n\le2^{t+1}-2$ (for some integer $t\ge4$), then this lower bound is actually the exact value of $\zcl_{\mathbb Z_2}(\widetilde G_{n,3})$; while if $2^t+2^{t-1}/3+1<n\le2^t+2^{t-1}+2^{t-2}$, then $\zcl_{\mathbb Z_2}(\widetilde G_{n,3})$ is either equal to this lower bound or is greater than it by $1$. These results then provide lower bounds for $\mathrm{TC}(\widetilde G_{n,3})$ as listed in Table \ref{table:3}.

Let $W_n$ be the image of the map $p^*:H^*(G_{n,3};\mathbb Z_2)\rightarrow H^*(\widetilde G_{n,3};\mathbb Z_2)$ induced by the well-known two-fold covering $p:\widetilde G_{n,3}\rightarrow G_{n,3}$ (which forgets the orientation of a subspace). The main tool for establishing our results on $\zcl_{\mathbb Z_2}(\widetilde G_{n,3})$ is the Gr\"obner basis $F_n$ (obtained in \cite{ColovicPrvulovic})
for the ideal $I_n$ that determines $W_n$. Indeed, the additive basis $\mathcal B_n$ for $W_n$ (induced by this Gr\"obner basis $F_n$) and some
identities in $W_n$ obtained from certain elements of $F_n$, proved
to be essential for our calculations.

The organization of the paper is as follows. In Section \ref{prel} we establish a framework for the subsequent calculations and proofs. We first recall some basic facts concerning the cohomology of $\widetilde G_{n,3}$, and list a few results from \cite{ColovicPrvulovic}. These include some identities involving the polynomials that generate the ideals $I_n$, Gr\"obner bases for these ideals etc. We also present some new identities in this regard. In this section some general results concerning the zero-divisor cup-length are given as well. The main part of the paper is Section \ref{section zcl_Wn}, in which we compute the exact value of $\zcl(W_n)$ (see Theorem \ref{zcl Wn}). In Section \ref{comp} we make a comparison between $\zcl(W_n)$ and $\zcl_{\mathbb Z_2}(\widetilde G_{n,3})$, obtaining our results on $\zcl_{\mathbb Z_2}(\widetilde G_{n,3})$ and thus the lower bounds for the topological complexity of $\widetilde G_{n,3}$.

In the rest of the paper the $\mathbb Z_2$ coefficients for cohomology will be understood, and so we will abbreviate $H^*(\widetilde G_{n,3};\mathbb Z_2)$ to $H^*(\widetilde G_{n,3})$, and $\zcl_{\mathbb Z_2}(\widetilde G_{n,3})$ to $\zcl(\widetilde G_{n,3})$.

\section{Preliminaries}
\label{prel}

\subsection{Background on cohomology algebra $H^*(\widetilde G_{n,3})$}

Let $n\ge6$ be an integer and $\widetilde G_{n,3}$ the Grassmann manifold consisting of oriented $3$-dimensional subspaces of the vector space $\mathbb R^n$. Since the cohomology algebra $H^*(G_{n,3})$ of the corresponding "unoriented" Grassmannian $G_{n,3}$ is generated by the Stiefel--Whitney classes of the canonical vector bundle over $G_{n,3}$, which pulls back via $p:\widetilde G_{n,3}\rightarrow G_{n,3}$ to the canonical vector bundle over $\widetilde G_{n,3}$, and since $\widetilde G_{n,3}$ is simply connected, the subalgebra $W_n=\im p^*$ of $H^*(\widetilde G_{n,3})$ is generated by the Stiefel--Whitney classes $\widetilde w_2$ and $\widetilde w_3$ of this canonical bundle. It is well known (see e.g.\ \cite{Fukaya}) that as a graded algebra
\begin{equation}\label{isomorphism}
W_n\cong\mathbb Z_2[w_2,w_3]/I_n,
\end{equation}
where $I_n$ is the homogeneous ideal in $\mathbb Z_2[w_2,w_3]$ generated by certain (homogeneous) polynomials $g_{n-2}$, $g_{n-1}$ and $g_n$ (the subscripts for both variables and polynomials indicate their degrees in $\mathbb Z_2[w_2,w_3]$).

The polynomials $g_r$, $r\ge0$, satisfy the equation
\begin{equation}\label{gps}
(1+w_2+w_3)(g_0+g_1+g_2+\cdots)=1
\end{equation}
(in the ring of power series $\mathbb Z_2[[w_2,w_3]]$), which leads to the recurrence formula:
\begin{equation}\label{recgpolk3}
g_{r+3}=w_2g_{r+1}+w_3g_r \quad \mbox{for all } r\ge0.
\end{equation}
This formula shows that the sequence of ideals $\{I_n\}_{n\ge2}$ is descending:
\[I_{n+1}=(g_{n-1},g_n,g_{n+1})=(g_{n-1},g_n,w_2g_{n-1}+w_3g_{n-2})\subseteq(g_{n-2},g_{n-1},g_n)=I_n;\]
and more generally, that
\[
g_r\in I_n \quad \mbox{ for all } r\ge n-2.
\]

The polynomials $g_r$ for small $r$ can be routinely calculated. In Table \ref{table:2} we list these polynomials for $0\le r\le26$.

\begin{table}[h!]
\footnotesize
\centering
\begin{tabular}{||m{0.9cm} m{3cm} ||}
\hline
 $r$ & $g_r$ \\ [0.5ex]
\hline\hline
 $0$ & $1$\\
\hline
 $1$ & $0$\\
\hline
 $2$ & $w_2$\\
\hline
 $3$ & $w_3$\\
\hline
 $4$ & $w_2^2$\\
\hline
 $5$ & $0$\\
\hline
 $6$ & $w_2^3+w_3^2$\\
\hline
 $7$ & $w_2^2w_3$\\
\hline
 $8$ & $w_2^4+w_2w_3^2$\\
\hline
 $9$ & $w_3^3$\\
\hline
 $10$ & $w_2^5$\\
\hline
 $11$ & $w_2^4w_3$\\
\hline
 $12$ & $w_2^6+w_3^4$\\
\hline
 $13$ & $0$\\
\hline
 $14$ & $w_2^7+w_2^4w_3^2+w_2w_3^4$\\
\hline
 $15$ & $w_2^6w_3+w_3^5$\\
\hline
 $16$ & $w_2^8+w_2^5w_3^2+w_2^2w_3^4$\\
\hline
 $17$ & $w_2^4w_3^3$\\
\hline
 $18$ & $w_2^9+w_2^3w_3^4+w_3^6$\\
\hline
 $19$ & $w_2^8w_3+w_2^2w_3^5$\\
\hline
 $20$ & $w_2^{10}+w_2w_3^6$\\
\hline
 $21$ & $w_3^7$\\
\hline
 $22$ & $w_2^{11}+w_2^8w_3^2$\\
\hline
 $23$ & $w_2^{10}w_3$\\
\hline
 $24$ & $w_2^{12}+w_2^9w_3^2+w_3^8$\\
\hline
 $25$ & $w_2^8w_3^3$\\
\hline
 $26$ & $w_2^{13}+w_2w_3^8$\\
\hline
\end{tabular}
\caption{}
\label{table:2}
\end{table}

From (\ref{gps}) it is not hard to deduce an explicit formula for $g_r$:
\begin{equation}\label{g-exp}
g_r=\sum_{2d+3e=r}\binom{d+e}{e}w_2^dw_3^e, \qquad r\ge0,
\end{equation}
where the sum is over all pairs $(d,e)$ of nonnegative integers such that $2d+3e=r$.

The coset of $w_i$ in the quotient $\mathbb Z_2[w_2,w_3]/I_n$ corresponds to the Stiefel--Whitney class $\widetilde w_i\in H^i(\widetilde G_{n,3})$, $i=2,3$, via the isomorphism (\ref{isomorphism}). This means that for every polynomial $f=f(w_2,w_3)\in\mathbb Z_2[w_2,w_3]$ the following equivalence holds:
\begin{equation}\label{ekv - s tildom - bez tilde}
f(w_2,w_3)\in I_n \quad \Longleftrightarrow \quad f(\widetilde w_2,\widetilde w_3)=0 \mbox{ in } W_n.
\end{equation}


Let us now recall the identities from \cite[Proposition 2.2]{ColovicPrvulovic}, which involve the polynomials $g_r$, $r\ge0$.

\begin{lemma}\label{g-3}
Let $t\ge2$ be an integer. Then:
\begin{itemize}
\item[(a)] $g_{2^t-3}=0$;
\item[(b)] $g_{2^t+2^{t-1}-3}=w_3^{2^{t-1}-1}$;
\item[(c)] $g_{2^t+2^{t-2}-3}=w_2^{2^{t-2}}w_3^{2^{t-2}-1}$;
\item[(d)] $g_{2^t+2^{t-1}+2^{t-2}-3}=w_2^{2^{t-1}}w_3^{2^{t-2}-1}$;
\item[(e)] $g_{2^t+2^{t-1}+2^{t-3}-3}=w_2^{2^{t-1}+2^{t-3}}w_3^{2^{t-3}-1}$ (if $t\ge3$).
\end{itemize}
\end{lemma}

We also state an important lemma from \cite[Lemma 4.3]{ColovicPrvulovic}, which establishes a nice property of the ideals $w_3I_n=\{w_3p\mid p\in I_n\}$, $n\ge2$.

\begin{lemma}\label{kvadriranje2}
Let $f\in\mathbb Z_2[w_2,w_3]$ and $n\ge2$. If $f\in w_3I_n$, then $f^2\in w_3I_{2n+1}$. In particular, the following implication holds:
\[f\in w_3I_n\,\Longrightarrow\, f^2\in w_3I_{2n}.\]
\end{lemma}

Note also that
\begin{equation}\label{w_3I_n subset I_n+1}
w_3I_n=(w_3g_{n-2},w_3g_{n-1},w_3g_n)\subseteq I_{n+1},
\end{equation}
since $w_3g_{n-2}=w_2g_{n-1}+g_{n+1}\in I_{n+1}$ (by (\ref{recgpolk3})).

We now state and prove some additional identities in $\mathbb Z_2[w_2,w_3]$.

\begin{lemma}\label{kvadriranje}
For all nonnegative integers $i$ and $r$ we have
\[g_{2^i(r+3)-3}=w_3^{2^i-1}g_r^{2^i}.\]
\end{lemma}

\begin{proof}
Our proof is by induction on $i\geq 0$. The case $i=0$ is trivial,
while the case $i=1$ is in fact
\cite[Lemma 2.1]{ColovicPrvulovic}.

So, we assume that the identity holds for some $i\ge1$ and prove it for $i+1$.
Then, using the base case and inductional hypothesis we get
\[g_{2^{i+1}(r+3)-3}=w_3g_{2^i(r+3)-3}^2=w_3\left(w_3^{2^i-1}g_r^{2^i}\right)^2=w_3^{2^{i+1}-1}g_r^{2^{i+1}},\]
and we are done.
\end{proof}

\begin{lemma}\label{2n preko n}
For every $n\geq 1$ one has:
\[g_{2n}=g_n^2+w_2g_{n-1}^2.\]
\end{lemma}

\begin{proof}
Our proof is by induction on $n$. By looking at Table \ref{table:2}, it is easy to check that the identity
holds for $n\in\{1,2,3\}$. So, suppose that it is true for all $m\in\{1,\ldots,n\}$
and let us prove it for $n+1\ge4$. By (\ref{recgpolk3}) and the inductional hypothesis
we have
\begin{align*}
g_{2n+2}&=w_2g_{2n}+w_3g_{2n-1}=w_2(w_2g_{2n-2}+w_3g_{2n-3})+w_3(w_2g_{2n-3}+w_3g_{2n-4})\\
&=w_2^2g_{2n-2}+w_3^2g_{2n-4}=w_2^2(g_{n-1}^2+w_2g_{n-2}^2)+w_3^2(g_{n-2}^2+w_2g_{n-3}^2)\\
&=(w_2g_{n-1}+w_3g_{n-2})^2+w_2(w_2g_{n-2}+w_3g_{n-3})^2=g_{n+1}^2+w_2g_{n}^2,
\end{align*}
which completes our proof.
\end{proof}

The subalgebra $W_n$ is strictly smaller than $H^*(\widetilde G_{n,3})$. In particular, it does not contain the nontrivial cohomology class in the top dimension (see e.g.\ \cite[p.\ 1171]{Korbas:ChRank}). Put in other words, since the dimension of the manifold $\widetilde G_{n,3}$ is $3n-9$, the following implication holds:
\begin{equation}\label{topdim}
\widetilde w_2^b\widetilde w_3^c\neq0\mbox{ in } H^*(\widetilde G_{n,3}) \quad \Longrightarrow \quad 2b+3c<3n-9.
\end{equation}

\subsection{Gr\"obner basis for the ideal $I_n$}

In \cite{ColovicPrvulovic} Gr\"obner bases for the ideals $I_n$, $n\ge7$, were obtained. These bases are with respect to the lexicographic monomial ordering in $\mathbb Z_2[w_2,w_3]$ in which $w_2$ is greater than $w_3$. So, the exponents of $w_2$ are first compared, and if they are equal, then one compares the exponents of $w_3$.

For an integer $n\ge7$ this Gr\"obner basis for the ideal $I_n=(g_{n-2},g_{n-1},g_n)$ is given as follows. First, $n$ is placed between two adjacent powers of two, more precisely: $2^t-1\le n<2^{t+1}-1$ for some integer $t\ge3$. Then one takes binary digits $\alpha_0,\alpha_1,\ldots,\alpha_{t-1}$ of the number $n-2^t+1$:
\[n-2^t+1=\sum_{j=0}^{t-1}\alpha_j2^j;\]
and for $-1\le i\le t-1$ defines $s_i:=\sum_{j=0}^i\alpha_j2^j$ (so $s_{-1}=0$). The Gr\"obner basis is now specified in the following theorem \cite[Theorem 3.14]{ColovicPrvulovic}.

\begin{theorem}[\cite{ColovicPrvulovic}]\label{Grebner}
The set $F_n=\{f_0,f_1,\ldots,f_{t-1}\}$ is a Gr\"obner basis for $I_n$ (with respect to the specified monomial ordering), where
\[f_i=w_3^{\alpha_is_{i-1}}g_{n-2+2^i-s_i}, \quad 0\le i\le t-1.\]
\end{theorem}


%
%

The leading monomials of the polynomials from $F_n$ are calculated in \cite[Proposition 3.9]{ColovicPrvulovic}:
\[
\mathrm{LM}(f_i)=w_2^{\frac{n+1-s_i}{2}-2^i}w_3^{\alpha_is_{i-1}+2^i-1}, \quad 0\le i\le t-1.
\]
Also, it is not hard to check (or to find in the proof of \cite[Proposition 3.9]{ColovicPrvulovic}) that $(n+1-s_i)/2-2^i$ is an integer divisible by $2^i$; more precisely,
\begin{equation}\label{LMgi}
\mathrm{LM}(f_i)=w_2^{2^il_i}w_3^{\alpha_is_{i-1}+2^i-1}, \quad 0\le i\le t-1,
\end{equation}
where
\[l_i=2^{t-1-i}+\sum_{j=i+1}^{t-1}\alpha_j2^{j-i-1}-1.\]

%
%

Let $\mathcal{B}_n$ be the set of all cohomology classes of the form $\widetilde w_2^b\widetilde w_3^c\in W_n\subset H^*(\widetilde G_{n,3})$ such that the corresponding monomial $w_2^bw_3^c$ is not divisible by any of the leading monomials $\mathrm{LM}(f_i)$, $0\le i\le t-1$. Then a well-known fact from the theory of Gr\"obner bases (together with the isomorphism (\ref{isomorphism}) and Theorem \ref{Grebner}) ensures that $\mathcal{B}_n$ is an additive basis of $W_n$.

The main purpose of Gr\"obner bases is deciding whether a polynomial belongs to the given ideal or not, i.e., whether its coset in the quotient ring is zero or not. This is done by reducing the polynomial using the elements of a Gr\"obner basis. In the following lemma (which will be used frequently in the next section) we describe the reduction of a monomial $w_2^bw_3^c$ by the polynomial $f_i\in F_n$.

\begin{lemma}\label{reduction - f_i}
Let $0\le i\le t-1$. If $b$ and $c$ are nonnegative integers such that $\mathrm{LM}(f_i)\mid w_2^bw_3^c$, i.e., if the monomial $w_2^bw_3^c$ can be reduced by $f_i$, then in $W_n$ the following equality holds:
\[\widetilde w_2^b\widetilde w_3^c=\sum_{\substack{2d+3e=2l_i\\ e>0}}\binom{d+e}{e}\widetilde w_2^{b-2^i(l_i-d)}\widetilde w_3^{c+2^ie},\]
where the sum is taken over all pairs of integers $(d,e)$ such that $d\ge0$, $e>0$ and $2d+3e=2l_i$.
\end{lemma}
\begin{proof}
By (\ref{ekv - s tildom - bez tilde}) it is enough to prove
\begin{equation}\label{cong}
w_2^bw_3^c\equiv\sum_{\substack{2d+3e=2l_i\\ e>0}}\binom{d+e}{e}w_2^{b-2^i(l_i-d)}w_3^{c+2^ie} \pmod{I_n}.
\end{equation}
In order to do so, let us note first that
\begin{equation}\label{f_i}
f_i=w_3^{\alpha_is_{i-1}}g_{2^i(2l_i+3)-3}.
 \end{equation}
This follows from the definition of $f_i$ (given in Theorem \ref{Grebner}) and the equality $n-2+2^i-s_i=2^i(2l_i+3)-3$, which can be easily verified (see also the proof of \cite[Proposition 3.9]{ColovicPrvulovic}).

By (\ref{LMgi}), the assumption $\mathrm{LM}(f_i)\mid w_2^bw_3^c$ actually means that $b\ge2^il_i$ and $c\ge\alpha_is_{i-1}+2^i-1$. Now we use (\ref{f_i}), Lemma \ref{kvadriranje} and formula (\ref{g-exp}) to calculate:
    \begingroup
    \allowbreak
     \begin{align*}
         w_2^bw_3^c&\equiv w_2^{b-2^il_i}w_3^{c-\left(\alpha_is_{i-1}+2^i-1\right)}\left(f_i+w_2^{2^il_i}w_3^{\alpha_is_{i-1}+2^i-1}\right)\\
               &=w_2^{b-2^il_i}w_3^{c-\left(\alpha_is_{i-1}+2^i-1\right)}\left(w_3^{\alpha_is_{i-1}}g_{2^i(2l_i+3)-3}+w_2^{2^il_i}w_3^{\alpha_is_{i-1}+2^i-1}\right)\\
               &=w_2^{b-2^il_i}w_3^c\left(g_{2l_i}^{2^i}+w_2^{2^il_i}\right)=w_2^{b-2^il_i}w_3^c\left(g_{2l_i}+w_2^{l_i}\right)^{2^i}\\
               &=w_2^{b-2^il_i}w_3^c\sum_{\substack{2d+3e=2l_i\\ e>0}}
              \binom{d+e}{e}w_2^{2^id}w_3^{2^ie} \pmod{I_n},
     \end{align*}
     \endgroup
and (\ref{cong}) follows.
\end{proof}

\subsection{Background on zero-divisor cup-length}

Let $A$ be a graded commutative $\mathbb Z_2$-algebra with identity. The elements of the kernel of the multiplication map $A\otimes A\stackrel{.}\longrightarrow A$ are called \em zero-divisors. \em If $a\in A$ is an arbitrary element, then an obvious zero-divisor is
\[z(a):=a\otimes1+1\otimes a.\]
It is straightforward that for $a,b\in A$ one has
\begin{equation}\label{z(a+b)}
z(a+b)=z(a)+z(b),
\end{equation}
and note that the following rule holds:
\begin{equation}\label{z(ab)}
z(ab)=z(a)z(b)+(1\otimes b)z(a)+(1\otimes a)z(b).
\end{equation}
Namely,
\begin{align*}
z(a)z(b)&=(a\otimes 1+1\otimes a)(b\otimes1+1\otimes b)=ab\otimes1+1\otimes ab+a\otimes b+b\otimes a\\
        &=z(ab)+(1\otimes b)(a\otimes1+1\otimes a)+(1\otimes a)(b\otimes1+1\otimes b)\\
        &=z(ab)+(1\otimes b)z(a)+(1\otimes a)z(b).
\end{align*}

When $a=b$, (\ref{z(ab)}) simplifies to $z(a^2)=z(a)^2$, and this routinely generalizes to
\begin{equation}\label{z(a^stepen dvojke)}
z\big(a^{2^l}\big)=z(a)^{2^l} \qquad \mbox{for all } l\ge0.
\end{equation}

The zero-divisor cup-length of $A$, $\zcl(A)$, is the maximal number of zero-divisors of positive degree with nonzero product in $A\otimes A$. Since the ideal $\ker\big(A\otimes A\stackrel{.}\longrightarrow A\big)$ is generated by the elements $z(a)$, $a\in A$ \cite[Lemma 5.2]{CohenSuciu}, $\zcl(A)$ is reached by a product of the form $z(a_1)z(a_2)\cdots z(a_m)$, where $a_i\in A$, $1\le i\le m$, are elements of positive degree.

An element of positive degree in $A$ is \em indecomposable \em if it cannot be expressed as a polynomial in elements of smaller degree. Now, using (\ref{z(a+b)}) and (\ref{z(ab)}), we see that $\zcl(A)$ is in fact reached by a product $z(a_1)z(a_2)\cdots z(a_m)$, where $a_i\in A$, $1\le i\le m$, are indecomposable elements.

Recall that the \em height \em of an algebra element $x$ is the maximal integer $d\ge1$ such that $x^d\neq0$. We denote the height of $x$ by $\height(x)$. By \cite[Lemma 4.3]{Pavesic} for an element $a\in A$ the following implication holds:
\begin{equation}\label{htz}
2^t\le\height(a)<2^{t+1} \quad \Longrightarrow \quad \height(z(a))=2^{t+1}-1.
\end{equation}

The heights of the Stiefel--Whitney
classes $\widetilde{w}_2$ and $\widetilde{w}_3$ (in $H^*(\widetilde G_{n,3})$) are obtained in \cite[Theorems 1.2 and 1.3]{ColovicPrvulovic} (see also \cite[Table 1]{ColovicPrvulovic}). This
result, combined with (\ref{htz}), will be very important in our calculations.

\begin{theorem}[\cite{ColovicPrvulovic}]\label{heights}
Let $n\geq 7$ and $t\geq 3$ be integers such that $2^t-1\leq n<2^{t+1}-1$. Then
\begin{align*}
\height(\widetilde w_2)&=\begin{cases}
2^t-4,& 2^t-1\le n\le2^t+2^{t-1}\\
2^{t+1}-3\cdot2^s-1,& 2^{t+1}-2^{s+1}+1\le n\le2^{t+1}-2^s \quad (1\le s\le t-2)\\
\end{cases},\\
\height(\widetilde w_3)&=\max\{2^{t-1}-2,n-2^t-1\}.
\end{align*}
\end{theorem}

\section{The zero-divisor cup-length of $W_n$}
\label{section zcl_Wn}

In this section we prove our results on $\zcl(W_n)$. By definition of $W_n$, the only indecomposable elements in this algebra are $\widetilde w_2$ and $\widetilde w_3$, so $\zcl(W_n)$ is realized by a product of the form $z(\widetilde w_2)^\beta z(\widetilde w_3)^\gamma$. This fact will be used throughout the paper.

We begin by showing that $\zcl(W_n)$ increases with $n$.
\begin{lemma}\label{zcl raste}
For all integers $n\ge6$ one has
\[\zcl(W_n)\le\zcl(W_{n+1}).\]
\end{lemma}
\begin{proof}
Let $m=\zcl(W_{n+1})$. It suffices to prove that $z(\widetilde w_2)^\beta z(\widetilde w_3)^\gamma=0$ in $W_n\otimes W_n$, whenever $\beta+\gamma>m$.

Since $I_{n+1}\subseteq I_n$, the identity map on $\mathbb Z_2[w_2,w_3]$ induces the algebra morphism
\[\mathbb Z_2[w_2,w_3]/I_{n+1}\rightarrow\mathbb Z_2[w_2,w_3]/I_n,\]
and using the isomorphism (\ref{isomorphism}) we get a map $\phi:W_{n+1}\rightarrow W_n$ mapping $\widetilde w_i\in W_{n+1}$ to $\widetilde w_i\in W_n$, $i=2,3$. Note that then the algebra morphism
\[\phi\otimes\phi:W_{n+1}\otimes W_{n+1}\rightarrow W_n\otimes W_n\]
maps $z(\widetilde w_i)\in W_{n+1}\otimes W_{n+1}$ to $z(\phi(\widetilde w_i))=z(\widetilde w_i)\in W_n\otimes W_n$, $i=2,3$. However, if $\beta+\gamma>m$, then $z(\widetilde w_2)^\beta z(\widetilde w_3)^\gamma=0$ in $W_{n+1}\otimes W_{n+1}$ (since $\zcl(W_{n+1})=m$), and so
\[z(\widetilde w_2)^\beta z(\widetilde w_3)^\gamma=(\phi\otimes\phi)\big(z(\widetilde w_2)^\beta z(\widetilde w_3)^\gamma\big)=(\phi\otimes\phi)(0)=0 \quad \mbox{ in } W_n\otimes W_n,\]
completing the proof.
\end{proof}

Let us now state the main theorem of this section (cf.\ Table \ref{table:3}).
\begin{theorem}\label{zcl Wn}
Let $t\ge4$. If $2^t-1\leq n<2^{t+1}-1$, then
\[\zcl(W_n)=\left\{
\begin{array}{rl}
2^t+2^{t-1}-4, & 2^t-1\leq n\leq 2^t+2^{t-2} \\
2^t+2^{t-1}-3, & n=2^t+2^{t-2}+1 \\
2^t+2^{t-1}-2, & 2^t+2^{t-2}+2\leq n\leq 2^t+2^{t-1} \\
2^{t+1}+2^{t-3}-3, & n=2^t+2^{t-1}+1 \\
2^{t+1}+2^{t-3}-2, & 2^t+2^{t-1}+2\leq n\leq 13\cdot 2^{t-3}\\
2^{t+1}+2^{t-2}-2, &  13\cdot 2^{t-3}+1\leq n\leq 2^t+2^{t-1}+2^{t-2}\\
3\cdot2^{t}-2^{s+1}-2, & 2^{t+1}-2^{s+1}+1\leq n\leq 2^{t+1}-2^s,\\
&\phantom{aaaaaa}\mbox{ where }1\leq s\leq t-3
\end{array}\right..\]
\end{theorem}

\begin{remark}
Theorem \ref{zcl Wn} establishes the values of $\zcl(W_n)$ for all $n\ge15$. For small $n$, more precisely $6\le n\le14$,
the values of $\zcl(W_n)$ are given in
Table \ref{table:small n}. These are obtained by using the computer software SAGE, but they could be calculated "by hand" as well. For $6\le n\le11$ this was done in \cite{Ramani}.
\end{remark}

\begin{table}[h!]
\footnotesize
\centering
\begin{tabular}{||m{0.9cm} m{1cm} ||}
\hline
 $n$ & $\zcl(W_n)$ \\ [0.5ex]
\hline\hline
 $6$ & $2$\\
\hline
 $7$ & $7$\\
\hline
 $8$ & $7$\\
\hline
 $9$ & $7$\\
\hline
 $10$ & $8$\\
\hline
 $11$ & $9$\\
\hline
 $12$ & $10$\\
\hline
 $13$ & $15$\\
\hline
 $14$ & $16$\\
\hline
\end{tabular}
\caption{}
\label{table:small n}
\end{table}

The rest of this section is devoted to proving Theorem \ref{zcl Wn}. Throughout the proof we will use Lemma \ref{zcl raste}. For example, by this lemma, to prove
that $\zcl(W_n)=2^t+2^{t-1}-4$ for $2^t-1\leq n\leq 2^t+2^{t-2}$, it is enough to
prove that $\zcl(W_{2^t-1})\geq 2^t+2^{t-1}-4$ and
$\zcl(W_{2^t+2^{t-2}})\leq 2^t+2^{t-1}-4$. The results of the first type will be referred
to as ''the lower bound" and the results of the second type as ''the upper bound".

The lower bounds will be established by detecting a nonzero monomial of the form $z(\widetilde w_2)^\beta z(\widetilde w_3)^\gamma\in W_n\otimes W_n$. In proving that such a monomial is nonzero we will rely on the following lemma.

\begin{lemma}\label{z nonzero}
Let $\beta$, $\gamma$ and $r$ be nonnegative integers such that $r\le2\beta+3\gamma$. Observe the element
\begin{equation}\label{nonzero sum}
\sum_{2b+3c=r}\binom{\beta}{b}\binom{\gamma}{c}\widetilde w_2^b\widetilde w_3^c\otimes\widetilde w_2^{\beta-b}\widetilde w_3^{\gamma-c}
\end{equation}
in $W_n\otimes W_n\subset H^*(\widetilde G_{n,3})\otimes H^*(\widetilde G_{n,3})$, where the sum is taken over all pairs of integers $(b,c)$ such that $0\le b\le\beta$, $0\le c\le\gamma$ and $2b+3c=r$. If (\ref{nonzero sum}) is nonzero, then $z(\widetilde w_2)^\beta z(\widetilde w_3)^\gamma\neq0$, and so \[\zcl(W_n)\ge\beta+\gamma.\]
\end{lemma}
\begin{proof}
Since
\[z(\widetilde w_2)^\beta z(\widetilde w_3)^\gamma\in\big(H^*(\widetilde G_{n,3})\otimes H^*(\widetilde G_{n,3})\big)_{2\beta+3\gamma}=\bigoplus_{r=0}^{2\beta+3\gamma}H^r(\widetilde G_{n,3})\otimes H^{2\beta+3\gamma-r}(\widetilde G_{n,3}),\]
it suffices to find $r\in\{0,1,\ldots,2\beta+3\gamma\}$ such that the summand of $z(\widetilde w_2)^\beta z(\widetilde w_3)^\gamma$ in $H^r(\widetilde G_{n,3})\otimes H^{2\beta+3\gamma-r}(\widetilde G_{n,3})$ is nonzero. It remains to see that this summand is exactly (\ref{nonzero sum}), which is obvious from the following calculation:
\begin{align*}
z(\widetilde w_2)^\beta z(\widetilde w_3)^\gamma&=\big(\widetilde w_2\otimes1+1\otimes\widetilde w_2\big)^\beta\big(\widetilde w_3\otimes1+1\otimes\widetilde w_3\big)^\gamma\\
                                                &=\sum_{b=0}^\beta\binom{\beta}{b}\widetilde w_2^b\otimes\widetilde w_2^{\beta-b}\sum_{c=0}^\gamma\binom{\gamma}{c}\widetilde w_3^c\otimes\widetilde w_3^{\gamma-c}\\
                                                &=\sum_{b=0}^\beta\sum_{c=0}^\gamma\binom{\beta}{b}\binom{\gamma}{c}\widetilde w_2^b\widetilde w_3^c\otimes\widetilde w_2^{\beta-b}\widetilde w_3^{\gamma-c}.
\end{align*}
\end{proof}

Also, in the upcoming computations we will repeatedly use the following fact (which holds because we are working over a field): for $\sigma,\tau\in H^*(\widetilde G_{n,3})$ we have that $\sigma\otimes\tau\neq0$ in $H^*(\widetilde G_{n,3})\otimes H^*(\widetilde G_{n,3})$ if and only if $\sigma\neq0$ and $\tau\neq0$ in $H^*(\widetilde G_{n,3})$.

We break the proof of Theorem \ref{zcl Wn} in seven cases, as listed in the theorem, but the
strategy in each of them is quite similar.

\subsection{The case $\mathbf{2^t-1\leq n\leq 2^t+2^{t-2}}$}

In order to establish the lower bound in this case, we start off by identifying the elements of the Gr\"obner basis $F_{2^t-1}$ (from Theorem \ref{Grebner}) and their leading monomials.

\begin{lemma}\label{pomocna sa stepen -1}
For $n=2^{t}-1$, where $t\geq3$, we have:
\[f_i=g_{2^{t}+2^i-3}\quad\mbox{and}\quad \mathrm{LM}(f_i)=w_2^{2^{t-1}-2^i}w_3^{2^i-1},\quad\mbox{for }0\le i\le t-1.\]
\end{lemma}
\begin{proof}
We have $n-2^{t}+1=0$, and hence $\alpha_i=0$ and $s_i=0$ for all
$i$. So, the lemma follows from the definition of the
polynomials $f_i$ and (\ref{LMgi}).
\end{proof}

Now we consider the monomials of the form $\widetilde w_2^b\widetilde w_3^c$ in $H^{2^{t+1}-11}(\widetilde G_{2^t-1,3})$ and deduce which of them are nonzero.

\begin{lemma}\label{stepen -1}
     Let $n=2^{t}-1$, $t\ge3$. The only nonzero monomials of the form $\widetilde w_2^b\widetilde w_3^c$ in $H^{2^{t+1}-11}(\widetilde G_{n,3})$ are $\widetilde w_2^{2^{t}-3\cdot 2^{k-1}-1}\widetilde w_3^{2^k-3}$ for $2\leq k\leq t-1$. Furthermore, all of them are equal, i.e., for all $k\in\{2,\ldots,t-1\}$ we have the equality
    \[\widetilde w_2^{2^{t}-3\cdot 2^{k-1}-1}\widetilde w_3^{2^k-3}=\widetilde w_2^{2^{t-2}-1}\widetilde w_3^{2^{t-1}-3},\]
    and the monomial $\widetilde w_2^{2^{t-2}-1}\widetilde w_3^{2^{t-1}-3}$ is in the additive basis $\mathcal B_n$.
\end{lemma}
\begin{proof}
    By Lemma \ref{pomocna sa stepen -1}, none of $\mathrm{LM}(f_i)$, for $0\leq i\leq t-1$,
    divides $w_2^{2^{t-2}-1}w_3^{2^{t-1}-3}$, so $\widetilde w_2^{2^{t-2}-1}\widetilde w_3^{2^{t-1}-3}\in\mathcal B_n$. In particular, $\widetilde w_2^{2^{t-2}-1}\widetilde w_3^{2^{t-1}-3}\neq0$. Also, since $\height(\widetilde w_3)=2^{t-1}-2$ (by Theorem \ref{heights}), $\widetilde w_2^b\widetilde w_3^c=0$ if $c>2^{t-1}-2$.

    By using backward induction on $k$, where $2\leq k\leq t-1$, we prove that
    $\widetilde w_2^b\widetilde w_3^c$, with $2b+3c=2^{t+1}-11$ and $2^k-3\leq c<2^{k+1}-3$, is
    nonzero if and only if $c=2^k-3$, and that $\widetilde w_2^{2^{t}-3\cdot 2^{k-1}-1}\widetilde w_3^{2^k-3}=\widetilde w_2^{2^{t-2}-1}\widetilde w_3^{2^{t-1}-3}$. Note that $2b+3c=2^{t+1}-11$
    implies $c$ is odd, and we will use this throughout the proof.

    \medskip

    The induction base ($k=t-1$) follows from the first paragraph of the proof (and the fact that $c$ is odd).

    \medskip

    Now, let $2\le k\le t-2$ and $2^k-3\leq c\leq 2^{k+1}-5$.

    Suppose first that $c>2^k-3$. We are going to reduce the monomial $w_2^bw_3^c$ by the polynomial $f_{k}$. This can be done, because $\mathrm{LM}(f_{k})=w_2^{2^{t-1}-2^k}w_3^{2^{k}-1}$ (Lemma \ref{pomocna sa stepen -1}),
    \[b=\frac12\cdot\left(2^{t+1}-11-3c\right)\geq 2^{t}-3\cdot 2^k+2>2^{t-1}-2^k \quad\mbox{and}\quad c\ge 2^{k}-1.\]
    According to Lemma \ref{reduction - f_i} we have ($l_k=2^{t-1-k}-1$ since $\alpha_j=0$ for all $j$):
    \[\widetilde w_2^b\widetilde w_3^c=\sum_{\substack{2d+3e=2^{t-k}-2\\ e>0}}
              \binom{d+e}{e}\widetilde w_2^{b-2^k(2^{t-1-k}-1-d)}\widetilde w_3^{c+2^ke}.\]
    Notice that in each summand from this sum $e$ is even, and therefore $e\ge2$, which leads to
    $c+2^{k}e\ge c+2^{k+1}>2^{k+1}-3$. So we can apply inductive hypothesis to
    conclude that if the term $\widetilde w_2^{b-2^k(2^{t-1-k}-1-d)}\widetilde w_3^{c+2^{k}e}$ is nonzero, then
    $c+2^{k}e=2^m-3$ for some $m>k+1$. It follows that $c\equiv -3\pmod{2^{k+1}}$, which is
    impossible since $2^k-1\leq c\leq 2^{k+1}-5$. Therefore, all summands in the last
    sum are zero, and hence $\widetilde w_2^b\widetilde w_3^c=0$.

    Suppose now $c=2^k-3$. In this case we reduce $w_2^bw_3^c$ by $f_{k-1}$. Again, this is possible since $\mathrm{LM}(f_{k-1})=w_2^{2^{t-1}-2^{k-1}}w_3^{2^{k-1}-1}$ (Lemma
    \ref{pomocna sa stepen -1}),
    \[b=2^{t}-3\cdot 2^{k-1}-1> 2^{t-1}-2^{k-1}\quad\mbox{and}\quad
        c=2^{k}-3\ge 2^{k-1}-1.\]
    Using again Lemma \ref{reduction - f_i} we get:
   \[\widetilde w_2^b\widetilde w_3^c=\sum_{\substack{2d+3e=2^{t-k+1}-2\\ e>0}}
              \binom{d+e}{e}\widetilde w_2^{b-2^{k-1}(2^{t-k}-1-d)}\widetilde w_3^{c+2^{k-1}e}.\]
   Similarly as in the first case, we have that
   $c+2^{k-1}e\ge c+2^k= 2^{k+1}-3$. So, for a nonzero summand $\widetilde w_2^{b-2^{k-1}(2^{t-k}-1-d)}\widetilde w_3^{c+2^{k-1}e}$ one has $c+2^{k-1}e=2^m-3$, for some $m\geq k+1$. Then $e=2^{m-k+1}-2$ and
   $d+e=2^{t-k}-2^{m-k}$. If $m\geq k+2$, then $e\equiv 2\pmod 4$ and $d+e\equiv 0\pmod 4$,
   and hence, by Lucas' theorem, $\binom{d+e}{e}$ is zero. Therefore,
   only the term with $m=k+1$ is possibly nonzero. Then $e=2$ and $d+e=2^{t-k}-2$, and
   hence $\binom{d+e}{e}$ is nonzero by Lucas' theorem, so
   \[\widetilde w_2^{2^{t}-3\cdot 2^{k-1}-1}\widetilde w_3^{2^k-3}=\widetilde w_2^b\widetilde w_3^c=\widetilde w_2^{b-2^{k-1}(2^{t-k}-1-d)}\widetilde w_3^{c+2^{k-1}e}=\widetilde w_2^{2^{t}-3\cdot 2^k-1}\widetilde w_3^{2^{k+1}-3},\]
   which is equal to $\widetilde w_2^{2^{t-2}-1}\widetilde w_3^{2^{t-1}-3}$ by the induction hypothesis.
\end{proof}

Next, we prove a lemma (and its consequence) that will be used in subsequent cases as well.


\begin{lemma}\label{2/4 - lema}
For $t\geq 4$ we have
\[g_{3\cdot 2^{t-1}}+w_2^{3\cdot 2^{t-2}}+\sum_{k=1}^{t-3}w_2^{3\cdot 2^{k-1}}w_3^{2^{t-1}-2^k}\in w_3I_{2^t+2^{t-2}+2^{t-4}}.\]
\end{lemma}

\begin{proof}
Our proof is by induction on $t\geq 4$. By looking at Table \ref{table:2}, we see that for $t=4$ the claim is
\[w_3I_{21}\ni g_{24}+w_2^{12}+w_2^3w_3^{6}=w_2^9w_3^2+w_2^3w_3^6+w_3^8=w_3(w_2g_{19}+g_{21}),\]
which is true (since $g_{19},g_{21}\in I_{21}$). So, suppose that it is true for some
$t\geq 4$ and let us prove it for $t+1$. By Lemma \ref{kvadriranje2}, then we have
\begin{align*}
w_3I_{2^{t+1}+2^{t-1}+2^{t-3}}&\ni g_{3\cdot 2^{t-1}}^2+w_2^{3\cdot 2^{t-1}}+\sum_{k=1}^{t-3}w_2^{3\cdot 2^{k}}w_3^{2^{t}-2^{k+1}}\\
&=g_{3\cdot 2^{t-1}}^2+w_2^{3\cdot 2^{t-1}}+\sum_{k=2}^{t-2}w_2^{3\cdot 2^{k-1}}w_3^{2^{t}-2^{k}},
\end{align*}
and hence it is enough to prove
\[g_{3\cdot 2^t}+g_{3\cdot 2^{t-1}}^2+w_2^3w_3^{2^t-2}\in w_3I_{2^{t+1}+2^{t-1}+2^{t-3}}.\]
By Lemma \ref{2n preko n}, (\ref{recgpolk3}), Lemma \ref{g-3}(b), and Lemma
\ref{kvadriranje}, we have
\begingroup
\allowbreak
\begin{align*}
  g_{3\cdot 2^t}+g_{3\cdot 2^{t-1}}^2+w_2^3w_3^{2^t-2} & =w_2g_{3\cdot 2^{t-1}-1}^2+ w_2^3w_3^{2^t-2} \\
  &=w_2(w_2^2g_{3\cdot 2^{t-1}-3}^2+w_3^2g_{3\cdot 2^{t-1}-4}^2)+w_2^3w_3^{2^t-2}\\
  & =w_2w_3^2g_{3\cdot 2^{t-1}-4}^2=w_2w_3g_{3\cdot 2^t-5}.
\end{align*}
\endgroup
Since $3\cdot 2^t-5\geq 2^{t+1}+2^{t-1}+2^{t-3}-2$, we have
$g_{3\cdot 2^t-5}\in I_{2^{t+1}+2^{t-1}+2^{t-3}}$, which completes our proof.
\end{proof}

\begin{corollary}\label{2/4 - posledica}
For $t\ge3$ one has
\[w_2^{3\cdot 2^{t-2}}\equiv\sum_{k=1}^{t-3}w_2^{3\cdot 2^{k-1}}w_3^{2^{t-1}-2^k}\pmod{I_{2^t+2^{t-2}+2}}.\]
\end{corollary}
\noindent(Note that the left-hand side of this congruence corresponds to the (nonexisting) summand for $k=t-1$ in the sum on the right-hand side.)
\begin{proof}
For $t=3$ it is understood that the sum on the right-hand side is zero, and we actually need to check that $w_2^6\in I_{12}$. But $w_2^6=w_2g_{10}\in I_{12}$ (see Table \ref{table:2}).

For $t\ge4$ one has $2^t+2^{t-2}+2^{t-4}\ge2^t+2^{t-2}+1$, and so $w_3I_{2^t+2^{t-2}+2^{t-4}}\subseteq w_3I_{2^t+2^{t-2}+1}\subseteq I_{2^t+2^{t-2}+2}$ (by (\ref{w_3I_n subset I_n+1})). Now, Lemma \ref{2/4 - lema} gives us
\[w_2^{3\cdot 2^{t-2}}\equiv g_{3\cdot 2^{t-1}}+\sum_{k=1}^{t-3}w_2^{3\cdot 2^{k-1}}w_3^{2^{t-1}-2^k}\pmod{I_{2^t+2^{t-2}+2}}.\]
Moreover, $g_{3\cdot 2^{t-1}}\in I_{2^t+2^{t-2}+2}$ (since $3\cdot 2^{t-1}>2^t+2^{t-2}$), and we are done.
\end{proof}

Now we are ready to establish the lower bound for $\zcl(W_n)$ in this case.

\begin{proposition}\label{1/4 - donje}
    Let $2^{t}-1 \le n \le 2^t+2^{t-2}$, where $ t \ge 4$. Then
     \[\zcl(W_{n})\geq 2^{t}+2^{t-1}-4.\]
\end{proposition}
\begin{proof}
    Note that by Lemma \ref{zcl raste}, it is enough to prove the inequality for
    $n=2^{t}-1$. We apply Lemma \ref{z nonzero} for $\beta=2^t-1$, $\gamma=2^{t-1}-3$ and $r=2^{t+1}-11$. We thus need to prove that
    \[\sum_{2b+3c=2^{t+1}-11}
        \binom{2^{t}-1}{b}\binom{2^{t-1}-3}{c}
        \widetilde w_2^b\widetilde w_3^c\otimes
        \widetilde w_2^{2^{t}-1-b}\widetilde w_3^{2^{t-1}-3-c}\neq0\]
    in $H^*(\widetilde G_{n,3})\otimes H^*(\widetilde G_{n,3})$. By Lemma \ref{stepen -1}, we only need to consider the summands with
    $(b,c)=(2^t-3\cdot 2^{k-1}-1,2^k-3)$, for $2\leq k\leq t-1$, and for each of them
    $\widetilde w_2^b\widetilde w_3^c=\widetilde w_2^{2^{t-2}-1}\widetilde w_3^{2^{t-1}-3}$. So, by Lucas' theorem, equivalence (\ref{ekv - s tildom - bez tilde}), the fact $I_n=I_{2^t-1}\supseteq I_{2^t+2^{t-2}+2}$ and Corollary
    \ref{2/4 - posledica}, the last sum becomes
    \begin{align*}
        \widetilde w_2^{2^{t-2}-1}\widetilde w_3^{2^{t-1}-3}\otimes &\sum_{k=2}^{t-1}\widetilde w_2^{3\cdot 2^{k-1}}\widetilde w_3^{2^{t-1}-2^k}\\
        &=\widetilde w_2^{2^{t-2}-1}\widetilde w_3^{2^{t-1}-3}\otimes\big(\widetilde w_2^{3\cdot 2^{t-3}}\widetilde w_3^{2^{t-2}}+\widetilde w_2^{3}\widetilde w_3^{2^{t-1}-2}\big).
    \end{align*}
    We know that the first coordinate of the last tensor is nonzero, and since we are working over a field, it is enough to prove that the second coordinate is nonzero too.
    By Lemmas \ref{pomocna sa stepen -1} and \ref{g-3}(c), we have
    $f_{t-2}=g_{2^t+2^{t-2}-3}=w_2^{2^{t-2}}w_3^{2^{t-2}-1}$,
    and hence $\widetilde w_2^{3\cdot 2^{t-3}}\widetilde w_3^{2^{t-2}}=0$, while
    $\widetilde w_2^{3}\widetilde w_3^{2^{t-1}-2}\in\mathcal B_n$, since, by Lemma \ref{pomocna sa stepen -1},
    none of $\mathrm{LM}(f_i)$, for $0\leq i\leq t-1$, divides $w_2^{3}w_3^{2^{t-1}-2}$ (here we use the assumption $t\ge4$). This completes our proof.
\end{proof}

As the first step towards establishing the upper bound in this case, we display the elements of the Gr\"obner basis $F_n$ for $n=2^t+2^{t-2}$. Then, in Lemma \ref{2/4 - gornje} we prove a technical result, which is crucial for obtaining the upper bound. Also, since the corresponding proof in the next case goes along the same lines as this one, here we present them together by including the case $n=2^t+2^{t-2}+1$.

\begin{lemma}\label{2/4 - Grebner1}
Let $n=2^t+2^{t-2}+\varepsilon$, where $t\ge4$ and $\varepsilon\in\{0,1\}$.
Then for the elements of the Gr\"obner basis $F_n$ (from Theorem \ref{Grebner}) one has:
\begingroup
\allowbreak
\begin{itemize}
\item $f_0=g_{2^t+2^{t-2}+2\varepsilon-2}$, $\mathrm{LM}(f_0)=w_2^{2^{t-1}+2^{t-3}+\varepsilon-1}$;
\item  $f_i=w_3^{2^i-1}\left(g_{2^{t-i}+2^{t-i-2}-2}\right)^{2^i}$, $\mathrm{LM}(f_i)=w_2^{2^{t-1}+2^{t-3}-2^i}w_3^{2^i-1}$, for $1\leq i\leq t-3$;
\item  $f_{t-3}=w_2^{2^{t-1}}w_3^{2^{t-3}-1}+w_2^{2^{t-3}}w_3^{2^{t-2}+2^{t-3}-1}$;
\item  $f_{t-2}=w_2^{2^{t-2}}w_3^{2^{t-2}+\varepsilon}$, $f_{t-1}=w_3^{2^{t-1}-1}$.
\end{itemize}
\endgroup
\end{lemma}

\begin{proof}
We have $n-2^t+1=2^{t-2}+1+\varepsilon$, and hence $\alpha_2=\alpha_3=\dots=\alpha_{t-3}=0$,
$\alpha_{t-2}=1$, $\alpha_{t-1}=0$, and $s_1=s_2=\dots=s_{t-3}=1+\varepsilon$,
$s_{t-2}=s_{t-1}=2^{t-2}+1+\varepsilon$. Furthermore, $\alpha_0=s_0=1-\varepsilon$ and $\alpha_1=\varepsilon$.

It is now straightforward from the definition of $f_i$ (given in Theorem \ref{Grebner}) that $f_0=g_{2^t+2^{t-2}+2\varepsilon-2}$ and
\[f_i=g_{2^t+2^{t-2}+2^i-3}=w_3^{2^i-1}\left(g_{2^{t-i}+2^{t-i-2}-2}\right)^{2^i} \qquad\mbox{for }1\leq i\leq t-3,\]
by  Lemma \ref{kvadriranje}. In particular, for $i=t-3$ we get
\[f_{t-3}=w_3^{2^{t-3}-1}g_8^{2^{t-3}}=w_2^{2^{t-1}}w_3^{2^{t-3}-1}+w_2^{2^{t-3}}w_3^{2^{t-2}+2^{t-3}-1}\]
(see Table \ref{table:2}). Also,
\[
f_{t-2}=w_3^{1+\varepsilon}g_{2^t+2^{t-2}-3}=w_2^{2^{t-2}}w_3^{2^{t-2}+\varepsilon} \mbox{ and }
f_{t-1}=g_{2^t+2^{t-1}-3}=w_3^{2^{t-1}-1},\]
by Lemma \ref{g-3} (parts (c) and (b) respectively).

The statements about leading monomials follow from the fact $\mathrm{LM}(g_{2l})=w_2^l$ (which is easily seen from (\ref{g-exp})).
\end{proof}

\begin{lemma}\label{2/4 - gornje}
Let $t\ge4$ and $n=2^t+2^{t-2}+\varepsilon$, where $\varepsilon\in\{0,1\}$.
Then in $W_n\otimes W_n$ one has:
\begin{itemize}
  \item[(a)] if $\varepsilon=0$, then $z(\widetilde w_2)^{2^{t}-1}z(\widetilde w_3)^{2^{t-1}-2}=0$ and
  $z(\widetilde w_2)^{2^{t}-2}z(\widetilde w_3)^{2^{t-1}-1}=0$;
  \item[(b)] if $\varepsilon=1$, then $z(\widetilde w_2)^{2^{t}-1}z(\widetilde w_3)^{2^{t-1}-1}=0$.
\end{itemize}
\end{lemma}

\begin{proof}
Let us consider $z(\widetilde w_2)^{2^{t}-\varepsilon'}z(\widetilde w_3)^{2^{t-1}-\varepsilon''}$, where
$(\varepsilon',\varepsilon'')\in\{(1,2),(2,1)\}$ in part (a), while
$\varepsilon'=\varepsilon''=1$ in part (b). We want to prove that this element is zero in $W_n\otimes W_n$.

First, let us show that
\begin{equation}\label{prva}
z(\widetilde w_2)^{2^{t}-\varepsilon'}z(\widetilde w_3)^{2^{t-1}-\varepsilon''}=z(\widetilde w_2)^{2^{t-2}-\varepsilon'}z(\widetilde w_3)^{2^{t-3}-\varepsilon''}\cdot A,
\end{equation}
where
\[A=\sum_{k=1}^{t-3}\Big(\widetilde w_2^{3\cdot 2^{k-1}}\widetilde w_3^{2^{t-1}-2^k}\otimes\widetilde w_3^{3\cdot 2^{t-3}}+\widetilde w_3^{3\cdot 2^{t-3}}\otimes\widetilde w_2^{3\cdot 2^{k-1}}\widetilde w_3^{2^{t-1}-2^k}\Big).\]

We will use the relation
\begin{equation}\label{relation 2/4}
\widetilde w_2^{2^{t-1}}\widetilde w_3^{2^{t-2}-1}=0,
\end{equation}
which holds in $W_n$. Namely, according to Lemma \ref{2/4 - Grebner1}, $f_{t-1}=w_3^{2^{t-1}-1}$ and $f_{t-3}=w_2^{2^{t-1}}w_3^{2^{t-3}-1}+w_2^{2^{t-3}}w_3^{2^{t-2}+2^{t-3}-1}$, so $\widetilde w_2^{2^{t-1}}\widetilde w_3^{2^{t-2}-1}=\widetilde w_2^{2^{t-3}}\widetilde w_3^{2^{t-1}-1}=0$.

By (\ref{z(a^stepen dvojke)}) we have
\begin{align*}
z(\widetilde w_2)^{2^{t}-\varepsilon'}z(\widetilde w_3)^{2^{t-1}-\varepsilon''}\!\!&=z(\widetilde w_2)^{2^{t-1}-\varepsilon'}z(\widetilde w_2)^{2^{t-1}}z(\widetilde w_3)^{2^{t-2}-\varepsilon''}z(\widetilde w_3)^{2^{t-2}}\\
&=z(\widetilde w_2)^{2^{t-1}-\varepsilon'}z(\widetilde w_3)^{2^{t-2}-\varepsilon''}z\big(\widetilde w_2^{2^{t-1}}\big)z\big(\widetilde w_3^{2^{t-2}}\big)\\
&=z(\widetilde w_2)^{2^{t-1}-\varepsilon'}\!z(\widetilde w_3)^{2^{t-2}-\varepsilon''}\!\!\Big(\!\widetilde w_2^{2^{t-1}}\otimes\widetilde w_3^{2^{t-2}}\!\!+\widetilde w_3^{2^{t-2}}\otimes\widetilde w_2^{2^{t-1}}\!\Big),
\end{align*}
since $\widetilde w_2^{2^{t-1}}\widetilde w_3^{2^{t-2}}=0$ (by (\ref{relation 2/4})).
If we denote our element $z(\widetilde w_2)^{2^{t}-\varepsilon'}z(\widetilde w_3)^{2^{t-1}-\varepsilon''}$ by $x$, we thus have
\begingroup
\allowbreak
\begin{align*}
x&=z(\widetilde w_2)^{2^{t-2}-\varepsilon'}z(\widetilde w_3)^{2^{t-2}-\varepsilon''}z\big(\widetilde w_2^{2^{t-2}}\big)\Big(\widetilde w_2^{2^{t-1}}\otimes\widetilde w_3^{2^{t-2}}+\widetilde w_3^{2^{t-2}}\otimes\widetilde w_2^{2^{t-1}}\Big)\\
&=z(\widetilde w_2)^{2^{t-2}-\varepsilon'}z(\widetilde w_3)^{2^{t-2}-\varepsilon''}\Big(\widetilde w_2^{2^{t-1}+2^{t-2}}\otimes\widetilde w_3^{2^{t-2}}+\widetilde w_3^{2^{t-2}}\otimes\widetilde w_2^{2^{t-1}+2^{t-2}}\Big).
\end{align*}
\endgroup
The latter equality is due to the fact
\[y:=z(\widetilde w_2)^{2^{t-2}-\varepsilon'}z(\widetilde w_3)^{2^{t-2}-\varepsilon''}\Big(\widetilde w_2^{2^{t-1}}\otimes\widetilde w_2^{2^{t-2}}\widetilde w_3^{2^{t-2}}+ \widetilde w_2^{2^{t-2}}\widetilde w_3^{2^{t-2}}\otimes\widetilde w_2^{2^{t-1}}\Big)=0.\]
Namely, by Lemma \ref{2/4 - Grebner1}, $f_{t-2}=w_2^{2^{t-2}}w_3^{2^{t-2}+\varepsilon}$, so for
$\varepsilon=0$ we clearly have $y=0$. If $\varepsilon=1$ (in that case $\varepsilon'=\varepsilon''=1$), then the fact $\widetilde w_2^{2^{t-2}}\widetilde w_3^{2^{t-2}+1}=0$ and the expansion $z(\widetilde w_3)^{2^{t-2}-1}=\sum_{i+j=2^{t-2}-1}\widetilde w_3^i\otimes\widetilde w_3^j$ lead to
\[y=z(\widetilde w_2)^{2^{t-2}-1}\Big(\widetilde w_2^{2^{t-1}}\widetilde w_3^{2^{t-2}-1}\otimes\widetilde w_2^{2^{t-2}}\widetilde w_3^{2^{t-2}}+\widetilde w_2^{2^{t-2}}\widetilde w_3^{2^{t-2}}\otimes\widetilde w_2^{2^{t-1}}w_3^{2^{t-2}-1}\Big)=0,\]
by (\ref{relation 2/4}). Therefore,
\begingroup
\allowbreak
\begin{align*}
x&=z(\widetilde w_2)^{2^{t-2}-\varepsilon'}z(\widetilde w_3)^{2^{t-3}-\varepsilon''}z(\widetilde w_3^{2^{t-3}})\Big(\widetilde w_2^{2^{t-1}+2^{t-2}}\otimes\widetilde w_3^{2^{t-2}}+\widetilde w_3^{2^{t-2}}\otimes\widetilde w_2^{2^{t-1}+2^{t-2}}\Big)\\
&=z(\widetilde w_2)^{2^{t-2}-\varepsilon'}z(\widetilde w_3)^{2^{t-3}-\varepsilon''}\Big(\widetilde w_2^{3\cdot2^{t-2}}\otimes\widetilde w_3^{2^{t-2}+2^{t-3}}+\widetilde w_3^{2^{t-2}+2^{t-3}}\otimes\widetilde w_2^{3\cdot2^{t-2}}\Big),
\end{align*}
\endgroup
since $\widetilde w_2^{2^{t-1}+2^{t-2}}\widetilde w_3^{2^{t-3}}=\widetilde w_2^{2^{t-2}+2^{t-3}}\widetilde w_3^{2^{t-2}+2^{t-3}}=0$ (by using the Gr\"obner basis elements $f_{t-3}$ and $f_{t-2}$). Finally, (\ref{prva}) now follows from Corollary \ref{2/4 - posledica} (and the fact $I_{2^t+2^{t-2}+2}\subseteq I_{2^t+2^{t-2}+\varepsilon}$).

\medskip

Now, by (\ref{prva}) we have $x=x_1+x_2$, where
\[x_1:=z(\widetilde w_2)^{2^{t-2}-\varepsilon'}z(\widetilde w_3)^{2^{t-3}-\varepsilon''}\sum_{k=1}^{t-3}
\widetilde w_2^{3\cdot 2^{k-1}}\widetilde w_3^{2^{t-1}-2^k}\otimes \widetilde w_3^{3\cdot 2^{t-3}}\]
and
\[x_2:=z(\widetilde w_2)^{2^{t-2}-\varepsilon'}z(\widetilde w_3)^{2^{t-3}-\varepsilon''}\sum_{k=1}^{t-3}
\widetilde w_3^{3\cdot 2^{t-3}}\otimes \widetilde w_2^{3\cdot 2^{k-1}}\widetilde w_3^{2^{t-1}-2^k}.\]
We want to prove that $x_1=x_2$. If we expand $x_1$ by binomial formula, we obtain summands of the form
\[\binom{2^{t-2}-\varepsilon'}{i}\binom{2^{t-3}-\varepsilon''}{j}\widetilde w_2^{3\cdot 2^{k-1}+i}\widetilde w_3^{2^{t-1}-2^k+j}\otimes \widetilde w_2^{2^{t-2}-\varepsilon'-i}
\widetilde w_3^{2^{t-1}-\varepsilon''-j},\]
where $1\leq k\leq t-3$, $0\leq i\leq 2^{t-2}-\varepsilon'$ and $0\leq j\leq 2^{t-3}-\varepsilon''$.
Let us first observe
\[\sigma(k,i,j):=\widetilde w_2^{3\cdot 2^{k-1}+i}\widetilde w_3^{2^{t-1}-2^k+j}.\]
Note that $2^{t-1}-2^k+j\geq 2^{t-1}-2^{t-3}>2^{t-2}+\varepsilon$, so, if
$3\cdot 2^{k-1}+i\geq 2^{t-2}$, then $\sigma(k,i,j)=0$ (since, $\widetilde w_2^{2^{t-2}}\widetilde w_3^{2^{t-2}+\varepsilon}=0$). Therefore, we may assume $i\le2^{t-2}-3\cdot2^{k-1}-1$. Similarly,
if $2^{t-1}-2^k+j\geq 2^{t-1}-1$, then $\sigma(k,i,j)=0$ (since $\widetilde w_3^{2^{t-1}-1}=0$), so we can shrink the interval for $j$ as well: $0\le j\le2^k-2$. Moreover, for the same reason, in order for
\[\tau(k,i,j):=\widetilde w_2^{2^{t-2}-\varepsilon'-i}\widetilde w_3^{2^{t-1}-\varepsilon''-j}\]
to be nonzero, one must have $2^{t-1}-\varepsilon''-j\le2^{t-1}-2$, i.e., $j\ge2-\varepsilon''$. Finally, we conclude that
\[x_1=\sum_{k=1}^{t-3}\,\,\sum_{i=1-\varepsilon'}^{2^{t-2}-3\cdot2^{k-1}-1}\sum_{j=2-\varepsilon''}^{2^k-2}\binom{2^{t-2}-\varepsilon'}{i}\binom{2^{t-3}-\varepsilon''}{j}\sigma(k,i,j)\otimes\tau(k,i,j).\]
For technical reasons (which will be clear soon) we set the lower boundary for $i$ to be $1-\varepsilon'$, and it is understood that $\binom{2^{t-2}-2}{-1}=0$ (if $\varepsilon'=2$).

By the same token, the corresponding expansion of $x_2$ is the following:
\[x_2=\sum_{k=1}^{t-3}\,\,\sum_{\overline i=1-\varepsilon'}^{2^{t-2}-3\cdot2^{k-1}-1}\sum_{\overline j=2-\varepsilon''}^{2^k-2}\binom{2^{t-2}-\varepsilon'}{\overline i}\binom{2^{t-3}-\varepsilon''}{\overline j}\tau(k,\overline i,\overline j)\otimes\sigma(k,\overline i,\overline j).\]

The idea now is to note that the change of variables $\overline i:=2^{t-2}-\varepsilon'-3\cdot2^{k-1}-i$ and $\overline j:=2^k-\varepsilon''-j$ transforms the sum $x_1$ to the sum $x_2$, leading to the conclusion $x_1=x_2$. Obviously, $1-\varepsilon'\le i\le2^{t-2}-3\cdot2^{k-1}-1$ is equivalent to $1-\varepsilon'\le\overline i\le2^{t-2}-3\cdot2^{k-1}-1$, and likewise, $2-\varepsilon''\le j\le2^k-2$ is equivalent to $2-\varepsilon''\le\overline j\le2^k-2$. Also it is routine to check that $\sigma(k,i,j)=\tau(k,\overline i,\overline j)$ and $\tau(k,i,j)=\sigma(k,\overline i,\overline j)$, so it remains to establish the congruence
\[\binom{2^{t-2}-\varepsilon'}{i}\binom{2^{t-3}-\varepsilon''}{j}\equiv\binom{2^{t-2}-\varepsilon'}{\overline i}\binom{2^{t-3}-\varepsilon''}{\overline j}\pmod2.\]
Since for all $l\in\{0,1,\ldots,2^m-1\}$, $\binom{2^m-1}{l}\equiv1\pmod2$, and $\binom{2^m-2}{l}\equiv1\pmod2$ if and only if $l$ is even, this amounts to showing that if $\varepsilon'=2$, then $i$ and $\overline i$ are of the same parity, and if $\varepsilon''=2$, then $j$ and $\overline j$ are of the same parity. The latter implication is clear from $\overline j=2^k-\varepsilon''-j=2^k-2-j$. For the former one, from $\overline i=2^{t-2}-\varepsilon'-3\cdot2^{k-1}-i=2^{t-2}-2-3\cdot2^{k-1}-i$ we see that the only problem is the case $k=1$ (and $\varepsilon'=2$). However, this case is impossible (more precisely, then the summand for $k=1$ is zero in both sums), since $\varepsilon'=2$ implies $\varepsilon''=1$ (see the very beginning of the proof), and we would have $1=2-\varepsilon''\le j\le2^k-2=0$.
\end{proof}

We are finally able to conclude the proof of Theorem \ref{zcl Wn} in this case, by verifying the inequality $\zcl(W_n)\leq 2^t+2^{t-1}-4$ (the opposite inequality is proved in Proposition \ref{1/4 - donje}).

\begin{proposition}
For $t\geq 4$ and $2^t-1\leq n\leq 2^t+2^{t-2}$ one has
\[\zcl(W_n)\leq 2^t+2^{t-1}-4.\]
\end{proposition}

\begin{proof}
By Lemma \ref{zcl raste} it is enough to prove
$\zcl(W_{2^t+2^{t-2}})\leq 2^t+2^{t-1}-4$. Suppose to the contrary that for
$n=2^t+2^{t-2}$ there are integers $\beta,\gamma\ge0$ such that
\[z(\widetilde w_2)^\beta z(\widetilde w_3)^\gamma\neq0\quad\mbox{in } W_n\otimes W_n\quad \quad\mbox{and}\quad \beta+\gamma=2^t+2^{t-1}-3.\]
According to Theorem \ref{heights}, $\height(\widetilde w_2)=2^t-4$ and $\height(\widetilde w_3)=2^{t-1}-2$.
So, by (\ref{htz}),
$\height(z(\widetilde w_2))=2^t-1$ and
$\height(z(\widetilde w_3))=2^{t-1}-1$, which implies $\beta\le2^t-1$ and $\gamma\le2^{t-1}-1$. Hence
$(\beta,\gamma)\in\{(2^t-1,2^{t-1}-2),(2^t-2,2^{t-1}-1)\}$, which contradicts
Lemma \ref{2/4 - gornje}(a).
\end{proof}

\subsection{The case $\mathbf{n=2^t+2^{t-2}+1}$}

As we have already mentioned, the strategy of the proof is more or less the same in every case. So, in order to establish the lower bound, we use Lemma \ref{z nonzero}. We are going to pick the cohomology dimension $r$ in which we will be able to suitably sort all nonzero monomials (as we did in Lemma \ref{stepen -1} in the previous case). This will be done in Lemma \ref{cetvrtina plus 2}. However, similarly as in some points in the previous case, the corresponding claim for the next case is proved in literally the same way, so we include the case $n=2^t+2^{t-2}+2$ in Lemma \ref{cetvrtina plus 2}. For that, we will first need to identify the members of the Gr\"obner basis $F_{2^t+2^{t-2}+2}$ (for $F_{2^t+2^{t-2}+1}$ this was done in Lemma \ref{2/4 - Grebner1}).

\begin{lemma}\label{2/4 - Grebner}
Let $n=2^t+2^{t-2}+2$, where $t\ge4$.
Then for the elements of the Gr\"obner basis $F_n$ (from Theorem \ref{Grebner}) one has:
\begingroup
\allowbreak
\begin{itemize}
\item  $f_0=g_{2^t+2^{t-2}}$, $\mathrm{LM}(f_0)=w_2^{2^{t-1}+2^{t-3}}$;
\item  $f_1=w_3^2\left(g_{2^{t-1}+2^{t-3}-2}\right)^2$, $\mathrm{LM}(f_1)=w_2^{2^{t-1}+2^{t-3}-2}w_3^2$;
\item  $f_i=w_3^{2^i-1}\left(g_{2^{t-i}+2^{t-i-2}-2}\right)^{2^i}$, $\mathrm{LM}(f_i)=w_2^{2^{t-1}+2^{t-3}-2^i}w_3^{2^i-1}$, for $2\leq i\leq t-3$;
\item  $f_{t-3}=\begin{cases}
w_2^8w_3^2+w_2^2w_3^6, & \, t=4\\
w_2^{2^{t-1}}w_3^{2^{t-3}-1}+w_2^{2^{t-3}}w_3^{2^{t-2}+2^{t-3}-1}, & \, t\ge5
\end{cases}$;
\item  $f_{t-2}=w_2^{2^{t-2}}w_3^{2^{t-2}+2}$, $f_{t-1}=w_3^{2^{t-1}-1}$.
\end{itemize}
\endgroup
\end{lemma}

\begin{proof}
We have $n-2^t+1=2^{t-2}+2+1$, and hence $\alpha_0=\alpha_1=1$, $\alpha_2=\alpha_3=\dots=\alpha_{t-3}=0$,
$\alpha_{t-2}=1$, $\alpha_{t-1}=0$, and $s_0=1$, $s_1=s_2=\dots=s_{t-3}=3$,
$s_{t-2}=s_{t-1}=2^{t-2}+3$. Therefore, the polynomials $f_i$, for $2\le i\le t-1$, can be obtained as in the proof of Lemma \ref{2/4 - Grebner1} by putting $\varepsilon=2$. The exceptions are $f_0$ and $f_1$, which we calculate by definition (see Theorem \ref{Grebner}) and Lemma \ref{kvadriranje}:
\[f_0=g_{2^t+2^{t-2}}, \qquad f_1=w_3g_{2^t+2^{t-2}-1}=w_3^2\left(g_{2^{t-1}+2^{t-3}-2}\right)^2.\]
The leading monomials are again obtained from the fact $\mathrm{LM}(g_{2l})=w_2^l$.
\end{proof}

\begin{lemma}\label{cetvrtina plus 2}
     Let $n=2^{t}+2^{t-2}+\varepsilon$, where $t\ge5$ and $\varepsilon\in\{1,2\}$.
     Then the only nonzero monomials of the form $\widetilde w_2^b\widetilde w_3^c$ in
     $H^{2^{t+1}-8}(\widetilde G_{n,3})$ are
    $\widetilde w_2^{2^{t}-3\cdot 2^{k-1}-1}\widetilde w_3^{2^k-2}$ for $1\leq k\leq t-1$. Furthermore, they are all equal, i.e.,
    \[\widetilde w_2^{2^{t}-3\cdot 2^{k-1}-1}\widetilde w_3^{2^k-2}=\widetilde w_2^{2^{t-2}-1}\widetilde w_3^{2^{t-1}-2}\quad
    \mbox{for }1\leq k\leq t-1,\]
    and $\widetilde w_2^{2^{t-2}-1}\widetilde w_3^{2^{t-1}-2}\in\mathcal B_n$.
\end{lemma}

\begin{proof}
    It is obvious by Lemmas \ref{2/4 - Grebner1} and \ref{2/4 - Grebner} that $w_2^{2^{t-2}-1}w_3^{2^{t-1}-2}$ is not divisible by any of $\mathrm{LM}(f_i)$ for $0\leq i\leq t-1$, so $\widetilde w_2^{2^{t-2}-1}\widetilde w_3^{2^{t-1}-2}\in\mathcal B_n$, and hence $\widetilde w_2^{2^{t-2}-1}\widetilde w_3^{2^{t-1}-2}\neq0$.

    According to Theorem \ref{heights}, $\height(\widetilde w_3)=2^{t-1}-2$, which means that $\widetilde w_2^b\widetilde w_3^c=0$ whenever $c>2^{t-1}-2$.

    By using backward induction on $k$, where $1\leq k\leq t-1$, we now prove that $\widetilde w_2^b\widetilde w_3^c$, with $2b+3c=2^{t+1}-8$ and $2^k-2\leq c\leq 2^{k+1}-3$, is nonzero if and only if $c=2^k-2$, and that
    $\widetilde w_2^{2^{t}-3\cdot2^{k-1}-1}\widetilde w_3^{2^k-2}=\widetilde w_2^{2^{t-2}-1}\widetilde w_3^{2^{t-1}-2}$.
    Note that $2b+3c=2^{t+1}-8$ implies that $c$ must be even.

    \medskip

    We have already established this claim for $k=t-1$.

    \medskip

    Now we deal with the case $k=t-2$,
    i.e., $2^{t-2}-2 \le c\le 2^{t-1}-4$. Firstly, suppose $c\geq 2^{t-2}+2$. Then
    $2b=2^{t+1}-8-3c$ implies $b\geq 2^{t-2}+2$, and hence
    $w_{2}^bw_{3}^c$ is divisible by $f_{t-2}=w_{2}^{2^{t-2}}w_{3}^{2^{t-2}+\varepsilon}$ implying $\widetilde w_{2}^b\widetilde w_{3}^c=0$.

    Secondly, let $c=2^{t-2}-2+2\delta$, where
    $\delta \in \{0,1\}$. Then $b=2^t-3\cdot2^{t-3}-1-3\delta$, and note that $b\ge2^{t-1}$ because $t\ge5$. Since $f_{t-3}\in I_n$,
    by Lemmas \ref{2/4 - Grebner1} and \ref{2/4 - Grebner} we have $\widetilde w_{2}^{2^{t-1}}\widetilde w_{3}^{2^{t-3}-1}=\widetilde w_{2}^{2^{t-3}}\widetilde w_{3}^{2^{t-2}+2^{t-3}-1}$, and hence:
    \[\widetilde w_{2}^b\widetilde w_{3}^c=\widetilde w_{2}^{b-2^{t-1}}\widetilde w_{3}^{c-2^{t-3}+1}\cdot \widetilde w_{2}^{2^{t-3}}\widetilde w_{3}^{2^{t-2}+2^{t-3}-1}=\widetilde w_{2}^{2^{t-2}-1-3\delta}\widetilde w_{3}^{2^{t-1}-2+2\delta}.\]
    For $\delta=1$ we get zero because $\height(\widetilde w_3)=2^{t-1}-2$, and for $\delta=0$ we get the desired nonzero class. This finishes the case $k=t-2$.

    \medskip

    Now let $t-3\ge k\ge1$, and take $c$ such that $2^k-2\le c\le 2^{k+1}-4$. We have two cases.
    Suppose first that $c\ge2^k$ (then $ k\geq 2$). We will reduce
    $w_2^bw_3^c$ by $f_k$. By Lemmas \ref{2/4 - Grebner1} and \ref{2/4 - Grebner} we know that $\mathrm{LM}(f_{k})=w_2^{2^{t-1}+2^{t-3}-2^k}w_3^{2^k-1}$, and since $2b=2^{t+1}-8-3c\geq 2^{t+1}-8-3\cdot(2^{ k+1}-4)$, we
    have
    \[b>2^{t-1}+2^{t-3}-2^k\quad\mbox{and}\quad c\ge2^k,\]
    which means that we can indeed reduce $w_2^bw_3^c$ by $f_k$, i.e., apply Lemma \ref{reduction - f_i} for $i=k$. It is obvious from (\ref{LMgi}) and the fact $\mathrm{LM}(f_{k})=w_2^{2^{t-1}+2^{t-3}-2^k}w_3^{2^k-1}$ that $l_k=2^{t-1-k}+2^{t-3-k}-1$, and so (by Lemma \ref{reduction - f_i}):
    \[\widetilde w_2^b\widetilde w_3^c=\sum_{\substack{2d+3e=2^{t-k}+2^{t-2-k}-2\\ e>0}}\binom{d+e}{e}\widetilde w_2^{b-2^k\left(2^{t-1-k}+2^{t-3-k}-1-d\right)}\widetilde w_3^{c+2^ke}.\]
    Obviously, $e$ must be even in every summand, and so $e\ge2$, which leads to
    $c+2^ke\ge c+2^{k+1}>2^{k+1}-2$. Hence, we can apply inductive hypothesis to
    conclude that if the term $\widetilde w_2^{b-2^{t-1}-2^{t-3}+2^k(d+1)}\widetilde w_3^{c+2^{k}e}$
    is nonzero, we have $c+2^{ k}e=2^m-2$ for some $m>k+1$.
    This implies  $c\equiv -2\pmod{2^{ k+1}}$, which is false, because $2^k\leq c\leq 2^{k+1}-4$. Therefore, every summand in the last sum is zero, and hence
    $\widetilde w_2^b\widetilde w_3^c=0$.

    Suppose now that $c=2^k-2$ (for $k\ge1$).
    We have $b=2^t-3\cdot2^{k-1}-1>2^{t-1}+2^{t-3}$, $\mathrm{LM}(f_0)=w_2^{2^{t-1}+2^{t-3}}$ (see Lemmas \ref{2/4 - Grebner1} and \ref{2/4 - Grebner}), so we can reduce $w_{2}^bw_{3}^c=w_{2}^{2^{t}-3\cdot 2^{k-1}-1}w_{3}^{2^k-2}$ by $f_0$, and by Lemma \ref{reduction - f_i} we get:
    \[\widetilde w_2^{2^t-3\cdot2^{k-1}-1}\widetilde w_3^{2^k-2}=\sum_{\substack{2d+3e=2^t+2^{t-2}\\ e>0}}\binom{d+e}{e}\widetilde w_2^{2^t-3\cdot2^{k-1}-1-\left(2^{t-1}+2^{t-3}-d\right)}\widetilde w_3^{2^k-2+e},\]
    since $l_0=2^{t-1}+2^{t-3}$. For every summand $2^k-2+e\ge2^k$, so we can apply inductive hypothesis to conclude
    that if the term $\widetilde w_2^{3\cdot 2^{t-3}-3\cdot2^{k-1}-1+d}\widetilde w_3^{2^k-2+e}$ is nonzero, it must be $2^k-2+e=2^m-2$, i.e., $e=2^m-2^k$, and consequently $d=2^{t-1}+2^{t-3}-3\cdot2^{m-1}+3\cdot2^{k-1}$, for some $m$ such that $k+1\le m\le t-1$.
    Thus if we single out only possibly nonzero summands in the above sum, we obtain
    \begin{equation*}
    \widetilde w_2^{2^t-3\cdot2^{k-1}-1}\widetilde w_3^{2^k-2}=\sum_{m=k+1}^{t-1}\binom{2^{t-1}+2^{t-3}-2^{m-1}+2^{k-1}}{2^m-2^k}\widetilde w_2^{2^t-3\cdot2^{m-1}-1}\widetilde w_3^{2^m-2}.
    \end{equation*}

    If $k=t-3$, then we have two summands in this sum (for $m=t-2$ and $m=t-1$), and the corresponding binomial coefficients are $\binom{2^{t-1}+2^{t-4}}{2^{t-3}}\equiv0\pmod2$ and $\binom{2^{t-2}+2^{t-3}+2^{t-4}}{2^{t-2}+2^{t-3}}\equiv1\pmod2$. Therefore, $\widetilde w_2^{2^t-3\cdot2^{k-1}-1}\widetilde w_3^{2^k-2}=\widetilde w_2^{2^{t-2}-1}\widetilde w_3^{2^{t-1}-2}$.

    If $k\le t-4$, let us prove that the only nonzero summand in the above sum is the one for $m=k+1$. Namely, for $m\ge k+2$ we have $2^{t-1}+2^{t-3}-2^{m-1}+2^{k-1}\equiv2^{k-1}\pmod{2^{k+1}}$, while $2^m-2^k\equiv2^k\pmod{2^{k+1}}$, and by Lucas' theorem the corresponding binomial coefficient vanishes. For $m=k+1$ we obtain $\binom{2^{t-1}+2^{t-3}-2^{k-1}}{2^k}\widetilde w_2^{2^t-3\cdot2^k-1}\widetilde w_3^{2^{k+1}-2}=\widetilde w_2^{2^t-3\cdot2^k-1}\widetilde w_3^{2^{k+1}-2}$ (again by Lucas' theorem). Finally, by the induction hypothesis, this monomial is equal to $\widetilde w_2^{2^{t-2}-1}\widetilde w_3^{2^{t-1}-2}$. This concludes the induction step and the proof of the lemma.
\end{proof}


We now prove the lower bound in this case.

\begin{proposition}\label{2/4 - donje1}
    Let $t\ge4$. Then $\zcl(W_{2^t+2^{t-2}+1})\ge2^t+2^{t-1}-3$.
\end{proposition}
\begin{proof}
By Lemma \ref{z nonzero} (for $\beta=2^t-1$, $\gamma=2^{t-1}-2$ and $r=2^{t+1}-8$) it is enough to prove that
\begin{equation}\label{1/4+1}
\sum_{2b+3c=2^{t+1}-8}
        \binom{2^{t}-1}{b}\binom{2^{t-1}-2}{c}
        \widetilde w_2^b\widetilde w_3^c\otimes
        \widetilde w_2^{2^{t}-1-b}\widetilde w_3^{2^{t-1}-2-c}
\end{equation}
is nonzero in $W_{2^t+2^{t-2}+1}\otimes W_{2^t+2^{t-2}+1}$.

If $t\ge5$, then by Lemma \ref{cetvrtina plus 2}, we only need to consider the summands with
    $(b,c)=(2^t-3\cdot 2^{k-1}-1,2^k-2)$, for $1\leq k\leq t-1$, and for each of them, we know that
    $\widetilde w_2^b\widetilde w_3^c=\widetilde w_2^{2^{t-2}-1}\widetilde w_3^{2^{t-1}-2}$. So, by Lucas' theorem and Corollary \ref{2/4 - posledica} (along with the fact $I_{2^t+2^{t-2}+2}\subseteq I_{2^t+2^{t-2}+1}$ and equivalence (\ref{ekv - s tildom - bez tilde})), (\ref{1/4+1}) becomes
    \begin{align*}
       \widetilde w_2^{2^{t-2}-1}\widetilde w_3^{2^{t-1}-2}\otimes \sum_{k=1}^{t-1}\widetilde w_2^{3\cdot 2^{k-1}}\widetilde w_3^{2^{t-1}-2^k}=\widetilde w_2^{2^{t-2}-1}\widetilde w_3^{2^{t-1}-2}\otimes \widetilde w_2^{2^{t-2}+2^{t-3}}\widetilde w_3^{2^{t-2}}.
    \end{align*}
    This last simple tensor is nonzero, because we already know that its first coordinate is in the additive basis
    $\mathcal B_{2^t+2^{t-2}+1}$, and it is routine to check (by Lemma \ref{2/4 - Grebner1}) that the second one belongs to $\mathcal B_{2^t+2^{t-2}+1}$ as well.

If $t=4$, then (\ref{1/4+1}) simplifies to
\begin{align*}
\sum_{2b+3c=24}
        \binom{15}{b}\binom{6}{c}
        \widetilde w_2^b\widetilde w_3^c\otimes
        \widetilde w_2^{15-b}\widetilde w_3^{6-c}=&\,\,\widetilde w_2^{12}\otimes
        \widetilde w_2^3\widetilde w_3^6+\widetilde w_2^9\widetilde w_3^2\otimes
        \widetilde w_2^6\widetilde w_3^4\\
        &+\widetilde w_2^6\widetilde w_3^4\otimes\widetilde w_2^9\widetilde w_3^2+\widetilde w_2^3\widetilde w_3^6\otimes\widetilde w_2^{12}.
\end{align*}
From Lemma \ref{2/4 - Grebner1} we see that the Gr\"obner basis $F_{21}$ consists of polynomials $f_0=g_{20}=w_2^{10}+w_2w_3^6$ (see Table \ref{table:2}), $f_1=w_2^8w_3+w_2^2w_3^5$, $f_2=w_2^4w_3^5$ and $f_3=w_3^7$. It is now routine to verify that $\widetilde w_2^3\widetilde w_3^6$ and $\widetilde w_2^6\widetilde w_3^4$ are two distinct elements of the additive basis $\mathcal B_{21}$, and that $\widetilde w_2^{12}=\widetilde w_2^9\widetilde w_3^2=\widetilde w_2^3\widetilde w_3^6$. Therefore, the above sum becomes
\[\widetilde w_2^3\widetilde w_3^6\otimes\widetilde w_2^6\widetilde w_3^4+\widetilde w_2^6\widetilde w_3^4\otimes\widetilde w_2^3\widetilde w_3^6,\]
and this is nonzero because $\widetilde w_2^3\widetilde w_3^6\otimes\widetilde w_2^6\widetilde w_3^4$ and $\widetilde w_2^6\widetilde w_3^4\otimes\widetilde w_2^3\widetilde w_3^6$ are two distinct elements of the additive basis $\{e\otimes f\mid e,f\in\mathcal B_{21}\}$ of $W_{21}\otimes W_{21}$.
\end{proof}

We are left to prove the upper bound in this case. Fortunately, we have everything prepared for this in the previous subsection (Lemma \ref{2/4 - gornje}(b)).

\begin{proposition}\label{2/4 - gornje1}
    Let $t\ge4$. Then $\zcl(W_{2^t+2^{t-2}+1})\le2^t+2^{t-1}-3$.
\end{proposition}
\begin{proof}
Let $z(\widetilde w_2)^\beta z(\widetilde w_3)^\gamma\in W_{2^t+2^{t-2}+1}\otimes W_{2^t+2^{t-2}+1}$, where $\beta+\gamma=2^t+2^{t-1}-2$. We need to prove that $z(\widetilde w_2)^\beta z(\widetilde w_3)^\gamma=0$.

Since $\height(z(\widetilde w_2))=2^t-1$ and $\height(z(\widetilde w_3))=2^{t-1}-1$ (by (\ref{htz}) and Theorem \ref{heights}), if $z(\widetilde w_2)^\beta z(\widetilde w_3)^\gamma$ were nonzero, we would have $\beta\le2^t-1$ and $\gamma\le2^{t-1}-1$. Together with $\beta+\gamma=2^t+2^{t-1}-2$ this leads to the conclusion $\beta=2^t-1$ and $\gamma=2^{t-1}-1$. But then we have a contradiction with Lemma \ref{2/4 - gornje}(b).
\end{proof}

Propositions \ref{2/4 - donje1} and \ref{2/4 - gornje1} prove Theorem \ref{zcl Wn} in the case $n=2^t+2^{t-2}+1$.

\subsection{The case $\mathbf{2^t+2^{t-2}+2\le n\le2^t+2^{t-1}}$}

In this case, for the upper bound the roughest estimate will do. Namely, $\height(\widetilde w_2)=2^t-4$ and $\height(\widetilde w_3)\in\{2^{t-1}-2,2^{t-1}-1\}$ (see Theorem \ref{heights}), and by (\ref{htz}), $\height(z(\widetilde w_2))=2^t-1$ and $\height(z(\widetilde w_3))=2^{t-1}-1$. Therefore, if $z(\widetilde w_2)^\beta z(\widetilde w_3)^\gamma\neq0$, then $\beta+\gamma\le2^t-1+2^{t-1}-1$, and so
\[\zcl(W_n)\le2^t+2^{t-1}-2.\]
Hence, the following proposition finishes the proof of Theorem \ref{zcl Wn} in this case.

\begin{proposition}
    Let $2^t+2^{t-2}+2\le n\le2^t+2^{t-1}$, where $t\ge4$. Then
    \[\zcl(W_{n})\ge2^t+2^{t-1}-2.\]
\end{proposition}
\begin{proof}
The proof is very similar to the proof of Proposition \ref{2/4 - donje1}. First of all, according to Lemma \ref{zcl raste}, it is enough to prove the inequality for
    $n=2^t+2^{t-2}+2$. We do this by applying Lemma \ref{z nonzero} for $\beta=2^t-1$, $\gamma=2^{t-1}-1$ (by the discussion preceding the proposition, $\beta$ and $\gamma$ have to be exactly these ones) and $r=2^{t+1}-8$. So, we show that
    \begin{equation}\label{2/4 nonzero}
\sum_{2b+3c=2^{t+1}-8}
        \binom{2^{t}-1}{b}\binom{2^{t-1}-1}{c}
        \widetilde w_2^b\widetilde w_3^c\otimes
        \widetilde w_2^{2^{t}-1-b}\widetilde w_3^{2^{t-1}-1-c}
\end{equation}
is nonzero in $W_n\otimes W_n$.

In the case $t\ge5$ we use Lemma \ref{cetvrtina plus 2} to single out nonzero summands in (\ref{2/4 nonzero}). They are the ones with
    $(b,c)=(2^t-3\cdot 2^{k-1}-1,2^k-2)$, for $1\leq k\leq t-1$, and they are all equal to
    $\widetilde w_2^{2^{t-2}-1}\widetilde w_3^{2^{t-1}-2}$. So, (\ref{2/4 nonzero}) simplifies to
    \begin{align*}
       \widetilde w_2^{2^{t-2}-1}\widetilde w_3^{2^{t-1}-2}\otimes \sum_{k=1}^{t-1}\widetilde w_2^{3\cdot 2^{k-1}}\widetilde w_3^{2^{t-1}-2^k+1}=\widetilde w_2^{2^{t-2}-1}\widetilde w_3^{2^{t-1}-2}\otimes \widetilde w_2^{2^{t-2}+2^{t-3}}\widetilde w_3^{2^{t-2}+1},
    \end{align*}
    by Corollary \ref{2/4 - posledica}. It remains to verify that $\widetilde w_2^{2^{t-2}-1}\widetilde w_3^{2^{t-1}-2},\widetilde w_2^{2^{t-2}+2^{t-3}}\widetilde w_3^{2^{t-2}+1}\in\mathcal B_n$, which easily follows from Lemma \ref{2/4 - Grebner}. So, (\ref{2/4 nonzero}) is nonzero if $t\ge5$.

In the case $t=4$ the sum (\ref{2/4 nonzero}) is equal to
\begin{align*}
\sum_{2b+3c=24}
        \binom{15}{b}\binom{7}{c}
        \widetilde w_2^b\widetilde w_3^c\otimes
        \widetilde w_2^{15-b}\widetilde w_3^{7-c}=&\,\,\widetilde w_2^{12}\otimes
        \widetilde w_2^3\widetilde w_3^7+\widetilde w_2^9\widetilde w_3^2\otimes
        \widetilde w_2^6\widetilde w_3^5\\
        &+\widetilde w_2^6\widetilde w_3^4\otimes\widetilde w_2^9\widetilde w_3^3+\widetilde w_2^3\widetilde w_3^6\otimes\widetilde w_2^{12}\widetilde w_3.
\end{align*}
Lemma \ref{2/4 - Grebner} gives us the Gr\"obner basis $F_{22}$ for the ideal $I_{22}$. It consists of polynomials $f_0=g_{20}=w_2^{10}+w_2w_3^6$, $f_1=w_2^8w_3^2+w_2^2w_3^6$, $f_2=w_2^4w_3^6$ and $f_3=w_3^7$. By using $f_0$, $f_1$ and $f_3$ we obtain $\widetilde w_2^{12}\widetilde w_3=\widetilde w_2^9\widetilde w_3^3=\widetilde w_2^3\widetilde w_3^7=0$. Also, $\widetilde w_2^9\widetilde w_3^2=\widetilde w_2^3\widetilde w_3^6$ (by using $f_1$), and so (\ref{2/4 nonzero}) is equal to $\widetilde w_2^3\widetilde w_3^6\otimes
        \widetilde w_2^6\widetilde w_3^5$. This simple tensor is nonzero because both $\widetilde w_2^3\widetilde w_3^6$ and $\widetilde w_2^6\widetilde w_3^5$ belong to the additive basis $\mathcal B_{22}$ (by the above Gr\"obner basis $F_{22}$).
\end{proof}

\subsection{The case $\mathbf{n=2^t+2^{t-1}+1}$}

In a similar fashion as before, in order to obtain the lower bound in this case, we start off with the Gr\"obner basis $F_{2^t+2^{t-1}+1}$, and then choose a cohomological dimension $r$ to apply Lemma \ref{z nonzero}.

\begin{lemma}\label{pomocna za pola plus 2}
    Let $n=2^t+2^{t-1}+1$, where $t\ge4$. Then for the elements $f_i$ ($0\le i\le t-1$) of the Gr\"obner basis $F_n$ one has:
    \begingroup
    \allowbreak
    \begin{itemize}
        \item $f_0=g_{2^t+2^{t-1}}$, $\mathrm{LM}(f_0)=w_2^{2^{t-1}+2^{t-2}}$;
        \item $f_i=w_3^{2^i-1}(g_{2^{t-i}+2^{t-1-i}-2})^{2^i}$,
            $\mathrm{LM}(f_i)=w_2^{2^{t-1}+2^{t-2}-2^i}w_3^{2^i-1}$ for $1\le i\le t-2$;
        \item $f_{t-2}=w_2^{2^{t-1}}w_3^{2^{t-2}-1}$, $f_{t-1}=w_3^{2^{t-1}+1}$.
    \end{itemize}
    \endgroup
\end{lemma}
\begin{proof}
We have $n-2^t+1=2^{t-1}+2$, and hence $\alpha_0=0$, $\alpha_1=1$, $\alpha_2=\alpha_3=\dots=\alpha_{t-2}=0$, $\alpha_{t-1}=1$, and
$s_0=0$, $s_1=s_2=\dots=s_{t-2}=2$, and
$s_{t-1}=2^{t-1}+2$. Since $f_i=w_3^{\alpha_is_{i-1}}g_{n-2+2^i-s_i}$, this implies
$f_0=g_{2^t+2^{t-1}}$,
\[f_i=g_{2^t+2^{t-1}-3+2^i}=w_3^{2^i-1}(g_{2^{t-i}+2^{t-1-i}-2})^{2^i}, \mbox{ for } 1\le i\le t-2,\]
by Lemma \ref{kvadriranje}. In particular,  $f_{t-2}=w_3^{2^{t-2}-1}g_4^{2^{t-2}}=w_2^{2^{t-1}}w_3^{2^{t-2}-1}$ (see Table \ref{table:2}).
The claims about the leading monomials are obvious consequences of $\mathrm{LM}(g_{2l})=w_2^l$ ($l\ge0$).
Finally, $f_{t-1}=w_3^2g_{2^t+2^{t-1}-3}=w_3^{2^{t-1}+1}$ by Lemma \ref{g-3}(b).
\end{proof}

The following lemma and its corollary will be used in this, as well as in the following two cases.

\begin{lemma}\label{treca cetvrtina - lema}
For $t\ge4$ we have
\[g_{2^{t+1}-6}+w_2^{2^{t}-3}+w_2^{2^{t-2}-3}w_3^{2^{t-1}}\in w_3I_{2^t+2^{t-1}+2^{t-3}+2^{t-4}}.\]
\end{lemma}

\begin{proof}
Our proof is by induction on $t\geq 4$. For $t=4$ the claim is
\[g_{26}+w_2^{13}+w_2w_3^{8}\in w_3I_{27},\]
which is clearly true, since $g_{26}+w_2^{13}+w_2w_3^8=0$ (see Table \ref{table:2}).

So, suppose that the claim is true for some
$t\geq 4$ and let us prove it for $t+1$. By Lemma \ref{kvadriranje2}:
\[g_{2^{t+1}-6}^2+w_2^{2^{t+1}-6}+w_2^{2^{t-1}-6}w_3^{2^{t}}\in w_3I_{2^{t+1}+2^{t}+2^{t-2}+2^{t-3}},\]
and hence
\[w_2^3g_{2^{t+1}-6}^2+w_2^{2^{t+1}-3}+w_2^{2^{t-1}-3}w_3^{2^{t}}\in w_3I_{2^{t+1}+2^{t}+2^{t-2}+2^{t-3}}.\]
So, it is enough to prove
\[g_{2^{t+2}-6}+w_2^3g_{2^{t+1}-6}^2\in w_3I_{2^{t+1}+2^{t}+2^{t-2}+2^{t-3}}.\]
By Lemma \ref{2n preko n}, Lemma \ref{g-3}(a), (\ref{recgpolk3}) and Lemma
\ref{kvadriranje}, we have
\begingroup
\allowbreak
\begin{align*}
  g_{2^{t+2}-6}+w_2^3g_{2^{t+1}-6}^2& =g_{2^{t+1}-3}^2+w_2g_{2^{t+1}-4}^2+w_2^3g_{2^{t+1}-6}^2\\
  &=w_2(w_2g_{2^{t+1}-6}+w_3g_{2^{t+1}-7})^2+w_2^3g_{2^{t+1}-6}^2\\
  &=w_2w_3^2g_{2^{t+1}-7}^2=w_2w_3g_{2^{t+2}-11}.
\end{align*}
\endgroup
Since $2^{t+2}-11>2^{t+1}+2^{t}+2^{t-2}+2^{t-3}-2$ for $t\geq 4$, we  have
$g_{2^{t+2}-11}\in I_{2^{t+1}+2^{t}+2^{t-2}+2^{t-3}}$, which completes our proof.
\end{proof}

\begin{corollary}\label{treca cetvrtina - posledica}
For $t\ge4$ one has
\[w_2^{2^{t}-3}\equiv w_2^{2^{t-2}-3}w_3^{2^{t-1}} \pmod{I_{13\cdot2^{t-3}+1}}.\]
\end{corollary}

\begin{proof}
For $t\ge4$ one clearly has $2^{t+1}-6\ge2^t+2^{t-1}+2^{t-3}+2^{t-4}-2$ and hence $g_{2^{t+1}-6}\in I_{2^t+2^{t-1}+2^{t-3}+2^{t-4}}$. Since $w_3I_n\subseteq I_n$, by Lemma \ref{treca cetvrtina - lema} we have
\[w_2^{2^{t}-3}+w_2^{2^{t-2}-3}w_3^{2^{t-1}}\in I_{2^t+2^{t-1}+2^{t-3}+2^{t-4}}\subseteq I_{2^t+2^{t-1}+2^{t-3}+1}=I_{13\cdot2^{t-3}+1},\]
and we are done.
\end{proof}

We are now ready for the proof of the lower bound.

\begin{proposition}\label{3/4 - donje1b - prvi}
    Let $t\ge4$. Then $\zcl(W_{2^t+2^{t-1}+1})\ge2^{t+1}+2^{t-3}-3$.
\end{proposition}
\begin{proof}
We apply Lemma \ref{z nonzero} for $\beta=2^{t+1}-1$, $\gamma=2^{t-3}-2$ and $r=2^{t+1}+2^{t-1}-2$. We need to show that
\[\sum_{2b+3c=2^{t+1}+2^{t-1}-2}\binom{2^{t+1}-1}{b}\binom{2^{t-3}-2}{c}\widetilde w_2^b\widetilde w_3^c\otimes\widetilde w_2^{2^{t+1}-1-b}\widetilde w_3^{2^{t-3}-2-c}\neq0.\]
Let us prove that all summands in this sum, except the one with $(b,c)=(2^{t}+2^{t-2}-1,0)$, are zero.

By Theorem \ref{heights}, $\height(\widetilde w_2)=2^t+2^{t-2}-1$, and so $\widetilde w_2^b\widetilde w_3^c=0$ if $b>2^{t}+2^{t-2}-1$.

If $b<2^{t}+2^{t-2}-1$, then $c>0$; but also, $c\le2^{t-3}-2$ implies $2b=2^{t+1}+2^{t-1}-2-3c\ge2^{t+1}+2^{t-3}+4$, i.e., $b\ge2^t+2^{t-4}+2>2^t-3$, and by Corollary \ref{treca cetvrtina - posledica} (using (\ref{ekv - s tildom - bez tilde}) and the fact $I_{13\cdot2^{t-3}+1}\subseteq I_{2^t+2^{t-1}+1}$) we have
\[\widetilde w_2^b\widetilde w_3^c=\widetilde w_2^{b-2^t+3}\widetilde w_2^{2^t-3}\widetilde w_3^c=\widetilde w_2^{b-2^t+3}\widetilde w_2^{2^{t-2}-3}\widetilde w_3^{2^{t-1}+c}=0,\]
because $\height(\widetilde w_3)=2^{t-1}$ (see Theorem \ref{heights}).

Therefore, by Lucas' theorem, the last sum becomes
    \begin{equation}\label{simple tensor}
        \widetilde w_2^{2^t+2^{t-2}-1}\otimes\widetilde w_2^{2^{t-1}+2^{t-2}}\widetilde w_3^{2^{t-3}-2}.
    \end{equation}
    The first coordinate of this simple tensor is nonzero because $\height(\widetilde w_2)=2^t+2^{t-2}-1$,
    so it suffices to prove $\widetilde w_2^{2^{t-1}+2^{t-2}}\widetilde w_3^{2^{t-3}-2}\neq0$.

    If $t=4$, then the simple tensor (\ref{simple tensor}) is $\widetilde w_2^{19}\otimes\widetilde w_2^{12}$, and since its first coordinate is nonzero, so is the second.

    If $t\ge5$, then by Lemma \ref{pomocna za pola plus 2} and Table \ref{table:2} we have
    \[f_{t-4}=w_3^{2^{t-4}-1}g_{22}^{2^{t-4}}=w_2^{2^{t-1}+2^{t-3}+2^{t-4}}w_3^{2^{t-4}-1}+w_2^{2^{t-1}}w_3^{2^{t-3}+2^{t-4}-1},\]
    and hence $\widetilde w_2^{2^{t-1}+2^{t-2}}\widetilde w_3^{2^{t-3}-2}=\widetilde w_2^{2^{t-1}+2^{t-4}}\widetilde w_3^{2^{t-2}-2}$. It is routine to check that
    $\widetilde w_2^{2^{t-1}+2^{t-4}}\widetilde w_3^{2^{t-2}-2}\in\mathcal B_{2^t+2^{t-1}+1}$ (by Lemma
    \ref{pomocna za pola plus 2}), which completes our proof.
\end{proof}

For the proof of the upper bound, in order to avoid unnecessary repeating of the same arguments, it will be convenient to include the next case as well. For that reason, we now establish upper bounds jointly for this case and the next one.

\begin{proposition}\label{3/4 - gornje1}
If $t\ge4$ and $2^t+2^{t-1}+1\leq n\leq 2^t+2^{t-1}+2^{t-3}$, then
\[\zcl(W_n)\le2^{t+1}+2^{t-3}-2.\]
Moreover,
\[\zcl(W_{2^t+2^{t-1}+1})\le2^{t+1}+2^{t-3}-3.\]
\end{proposition}

\begin{proof}
Let us prove the first inequality. Assume to the contrary that
\[x:=z(\widetilde w_2)^\beta z(\widetilde w_3)^\gamma\neq0 \quad\mbox{in }W_n\otimes W_n,\]
for some integers $\beta,\gamma\ge0$ such that $\beta+\gamma=2^{t+1}+2^{t-3}-1$.

Since $\height(\widetilde w_2)=2^t+2^{t-2}-1$ and $2^{t-1}\leq\height(\widetilde w_3)\le2^{t-1}+2^{t-3}-1$ (see Theorem \ref{heights}), by (\ref{htz}) we have
$\height(z(\widetilde w_2))=2^{t+1}-1$ and $\height(z(\widetilde w_3))= 2^{t}-1$. Hence,
$\beta\leq 2^{t+1}-1$ and $\gamma\le2^t-1$, and so $\gamma\ge2^{t-3}$ and $\beta\ge2^t+2^{t-3}>2^t$.
Further, $g_{2^t+2^{t-1}+2^{t-3}-2},g_{2^t+2^{t-1}+2^{t-3}}\in I_n$, and hence
\[w_3g_{2^t+2^{t-1}+2^{t-3}-3}=w_2g_{2^t+2^{t-1}+2^{t-3}-2}+g_{2^t+2^{t-1}+2^{t-3}}\in I_n.\]
By Lemma \ref{g-3}(e), this implies $\widetilde w_2^{2^{t-1}+2^{t-3}}\widetilde w_3^{2^{t-3}}=0$
in $W_n$. So, by (\ref{z(a^stepen dvojke)})
\begin{align*}
x&=z(\widetilde w_2)^{\beta-2^t}z(\widetilde w_3)^{\gamma-2^{t-3}}z\big(\widetilde w_2^{2^t}\big)z\big(\widetilde w_3^{2^{t-3}}\big)\\
&=z(\widetilde w_2)^{\beta-2^t}z(\widetilde w_3)^{\gamma-2^{t-3}}\big(\widetilde w_2^{2^t}\otimes 1+1\otimes\widetilde w_2^{2^t}\big)\big(\widetilde w_3^{2^{t-3}}\otimes 1+1\otimes\widetilde w_3^{2^{t-3}}\big)\\
&=z(\widetilde w_2)^{\beta-2^t}z(\widetilde w_3)^{\gamma-2^{t-3}}\big(\widetilde w_2^{2^t}\otimes\widetilde w_3^{2^{t-3}}+\widetilde w_3^{2^{t-3}}\otimes\widetilde w_2^{2^t}\big)\\
&=z(\widetilde w_3)^{\gamma-2^{t-3}}\sum_{i+j=\beta-2^t}\binom{\beta-2^t}{i}\widetilde w_2^i\otimes\widetilde w_2^j\big(\widetilde w_2^{2^t}\otimes\widetilde w_3^{2^{t-3}}+\widetilde w_3^{2^{t-3}}\otimes\widetilde w_2^{2^t}\big).
\end{align*}
Since $x\neq0$ there exists a pair of nonnegative integers $(i,j)$ (with $i+j=\beta-2^t$) such that the term $\widetilde w_2^i\otimes\widetilde w_2^j(\widetilde w_2^{2^t}\otimes\widetilde w_3^{2^{t-3}}+\widetilde w_3^{2^{t-3}}\otimes\widetilde w_2^{2^t})=\widetilde w_2^{2^t+i}\otimes\widetilde w_2^j\widetilde w_3^{2^{t-3}}+\widetilde w_2^i\widetilde w_3^{2^{t-3}}\otimes\widetilde w_2^{2^t+j}$ is nonzero. But then
$\min\{i,j\}\leq 2^{t-2}-1$ (since $\height(\widetilde w_2)=2^t+2^{t-2}-1$) and $\max\{i,j\}\leq 2^{t-1}+2^{t-3}-1$ (since $\widetilde w_2^{2^{t-1}+2^{t-3}}\widetilde w_3^{2^{t-3}}=0$). We conclude that
\[\beta-2^t=i+j=\min\{i,j\}+\max\{i,j\}\leq 2^{t-1}+2^{t-2}+2^{t-3}-2,\]
and then $\beta+\gamma=2^{t+1}+2^{t-3}-1$ implies that actually we must have $\gamma\ge2^{t-2}+1>2^{t-2}$. Now, as above we
get
\begin{equation}\label{x=z...}
x=z(\widetilde w_2)^{\beta-2^t}z(\widetilde w_3)^{\gamma-2^{t-2}}\big(\widetilde w_2^{2^t}\otimes\widetilde w_3^{2^{t-2}}+\widetilde w_3^{2^{t-2}}\otimes\widetilde w_2^{2^t}\big).
\end{equation}
We have $2^t+2^{t-1}+2^{t-2}-3\ge n-2$, and so $g_{2^t+2^{t-1}+2^{t-2}-3}\in I_n$, which by Lemma \ref{g-3}(d) (and (\ref{ekv - s tildom - bez tilde})) means that
$\widetilde w_2^{2^{t-1}}\widetilde w_3^{2^{t-2}-1}=0$ in $W_n$. Along with the fact $\height(\widetilde w_2)=2^t+2^{t-2}-1$ this gives us that
\[z(\widetilde w_2)^{2^{t-1}}\big(\widetilde w_2^{2^t}\otimes\widetilde w_3^{2^{t-2}}\!+\widetilde w_3^{2^{t-2}}\otimes\widetilde w_2^{2^t}\big)\!=\!\big(\widetilde w_2^{2^{t-1}}\otimes1+1\otimes\widetilde w_2^{2^{t-1}}\big)\big(\widetilde w_2^{2^t}\otimes\widetilde w_3^{2^{t-2}}\!+\widetilde w_3^{2^{t-2}}\otimes\widetilde w_2^{2^t}\big)\]
is equal to zero. In the light of (\ref{x=z...}) and the assumption $x\neq0$, this means that $\beta-2^t\le2^{t-1}-1$, which in turn implies $\gamma\ge2^{t-1}+2^{t-3}$ (because $\beta+\gamma=2^{t+1}+2^{t-3}-1$). Finally, we have
\[x=z(\widetilde w_2)^{\beta-2^t}z(\widetilde w_3)^{\gamma-2^{t-1}-2^{t-3}}\big(\widetilde w_2^{2^t}\otimes\widetilde w_3^{2^{t-1}}+\widetilde w_3^{2^{t-1}}\otimes\widetilde w_2^{2^t}\big)\big(\widetilde w_3^{2^{t-3}}\otimes1+1\otimes\widetilde w_3^{2^{t-3}}\big)=0,\]
since $\widetilde w_2^{2^{t-1}+2^{t-3}}\widetilde w_3^{2^{t-3}}=0$ and $\widetilde w_3^{2^{t-1}+2^{t-3}}=0$ ($\height(\widetilde w_3)\le2^{t-1}+2^{t-3}-1$). This con\-tradicts the assumption $x\neq0$, and thus concludes the proof of the first inequality.

\medskip

Let us now prove the second inequality. Suppose
to the contrary that there are integers $\beta,\gamma\ge0$ such that
\[y:=z(w_2)^\beta z(w_3)^\gamma\neq0 \quad\mbox{in }W_{2^t+2^{t-1}+1}\otimes W_{2^t+2^{t-1}+1},\]
and $\beta+\gamma=2^{t+1}+2^{t-3}-2$.

The same analysis as above now leads to $\gamma\ge2^{t-3}-1$ and $\beta\ge2^t+2^{t-3}-1>2^t$. If additionally $\gamma\ge2^{t-3}$ and $\beta\neq2^t+2^{t-1}-1$, then it is routine to check that we can use the same proof
as above to obtain a contradiction. So we are left with the cases $(\beta,\gamma)=(2^{t+1}-1,2^{t-3}-1)$ and $(\beta,\gamma)=(2^t+2^{t-1}-1,2^{t-1}+2^{t-3}-1)$.

Note that $g_{2^t+2^{t-1}+2^{t-3}-3}\in I_{2^t+2^{t-1}+1}$, so, by Lemma \ref{g-3}(e),
we have
\begin{equation}\label{nula-monom}
\widetilde w_2^{2^{t-1}+2^{t-3}}\widetilde w_3^{2^{t-3}-1}=0 \quad\mbox{in }W_{2^t+2^{t-1}+1}.
\end{equation}
Also, we know that $13\cdot2^{t-3}+1>12\cdot2^{t-3}+1=2^t+2^{t-1}+1$, and consequently $I_{13\cdot2^{t-3}+1}\subseteq I_{2^t+2^{t-1}+1}$, so Corollary \ref{treca cetvrtina - posledica} establishes the equality
\begin{equation}\label{posled}
\widetilde w_2^{2^{t}-3}=\widetilde w_2^{2^{t-2}-3}\widetilde w_3^{2^{t-1}} \quad\mbox{in }W_{2^t+2^{t-1}+1}.
\end{equation}

Let us consider the case $(\beta,\gamma)=(2^{t+1}-1,2^{t-3}-1)$. Since $\binom{2^l-1}{i}\equiv1\pmod2$ for all $i\in\{0,1,\ldots,2^l-1\}$, the binomial formula leads to
\begin{align*}
y&=\sum_{i=0}^{2^{t+1}-1}\,\,\sum_{j=0}^{2^{t-3}-1}\widetilde w_2^i\widetilde w_3^j\otimes\widetilde w_2^{2^{t+1}-1-i}\widetilde w_3^{2^{t-3}-1-j}\\
&=\sum_{i=2^{t-1}+2^{t-2}}^{2^t+2^{t-2}-1}\,\,\sum_{j=0}^{2^{t-3}-1}\widetilde w_2^i\widetilde w_3^j\otimes\widetilde w_2^{2^{t+1}-1-i}\widetilde w_3^{2^{t-3}-1-j},
\end{align*}
since $\height(\widetilde w_2)=2^t+2^{t-2}-1$. For $i\ge2^t-3$ the only possibly nonzero summand is the one for $j=0$. Namely, due to (\ref{posled}) and the fact $\height(\widetilde w_3)=2^{t-1}$ (see Theorem \ref{heights}) we have $\widetilde w_2^i\widetilde w_3^j=0$ if $j>0$. Similarly, if $i\le2^t+2$, then $2^{t+1}-1-i\ge2^t-3$, so $\widetilde w_2^{2^{t+1}-1-i}\widetilde w_3^{2^{t-3}-1-j}=0$ unless $j=2^{t-3}-1$. This means that for $i\in\{2^t-3,\ldots,2^t+2\}$ there are no nonzero summands, and that
\[y=\sum_{i=2^{t-1}+2^{t-2}}^{2^t-4}\widetilde w_2^i\widetilde w_3^{2^{t-3}-1}\otimes\widetilde w_2^{2^{t+1}-1-i}+\sum_{i=2^t+3}^{2^t+2^{t-2}-1}\widetilde w_2^i\otimes\widetilde w_2^{2^{t+1}-1-i}\widetilde w_3^{2^{t-3}-1}.\]
But both of these sums are zero due to (\ref{nula-monom}), contradicting the assumption $y\neq0$.

\medskip

Finally, we consider the case $(\beta,\gamma)=(2^t+2^{t-1}-1,2^{t-1}+2^{t-3}-1)$. Then
\begin{align*}
y&=z(\widetilde w_2)^{2^{t-1}-1}z(\widetilde w_3)^{2^{t-3}-1}\big(\widetilde w_2^{2^t}\otimes\widetilde w_3^{2^{t-1}}+\widetilde w_3^{2^{t-1}}\otimes\widetilde w_2^{2^t}\big)\\
&=z(\widetilde w_2)^{2^{t-1}-1}\sum_{i=0}^{2^{t-3}-1}\widetilde w_3^i\otimes\widetilde w_3^{2^{t-3}-1-i}\big(\widetilde w_2^{2^t}\otimes\widetilde w_3^{2^{t-1}}+\widetilde w_3^{2^{t-1}}\otimes\widetilde w_2^{2^t}\big)\\
&=z(\widetilde w_2)^{2^{t-1}-1}\big(\widetilde w_2^{2^t}\widetilde w_3^{2^{t-3}-1}\otimes\widetilde w_3^{2^{t-1}}+\widetilde w_3^{2^{t-1}}\otimes\widetilde w_2^{2^t}\widetilde w_3^{2^{t-3}-1}\big).
\end{align*}
The first equality holds because $\widetilde w_2^{2^t}\widetilde w_3^{2^{t-1}}=0$ (by (\ref{nula-monom})), and the third because $\height(\widetilde w_3)=2^{t-1}$. However, the expression in the brackets is zero (again by (\ref{nula-monom})), and so $y=0$. This contradiction concludes the proof of the proposition.
\end{proof}

Propositions \ref{3/4 - donje1b - prvi} and \ref{3/4 - gornje1} prove Theorem \ref{zcl Wn} in the case $n=2^t+2^{t-1}+1$.

\subsection{The case $\mathbf{2^t+2^{t-1}+2\le n\le13\cdot2^{t-3}}$}

We have already established the upper bound for this case in Proposition \ref{3/4 - gornje1}:
\[\zcl(W_n)\le2^{t+1}+2^{t-3}-2.\]
Therefore, the following proposition completes the proof of Theorem \ref{zcl Wn} in this case.

Let us first state a result from \cite[p.\ 282]{ColovicPrvulovic}, which will be used in this case, as well as in the next one: if $t$, $s$ and $n$ are integers such that $t\ge3$, $1\le s\le t-2$ and $2^{t+1}-2^{s+1}+1\le n\le 2^{t+1}-2^s$, then
\begin{equation}\label{cuplength0}
\widetilde w_2^{2^{t+1}-3\cdot2^s-1}\widetilde w_3^{n-2^{t+1}+2^{s+1}-1}\neq0 \quad\mbox{in } W_n\subset H^*(\widetilde G_{n,3}).
\end{equation}
In particular, for $s=t-2$ and $n=2^t+2^{t-1}+2$ we get
\begin{equation}\label{cuplength1}
\widetilde w_2^{2^t+2^{t-2}-1}\widetilde w_3\neq0 \quad\mbox{in }W_{2^t+2^{t-1}+2}.
\end{equation}

\begin{proposition}
    Let $2^t+2^{t-1}+2\le n\le2^t+2^{t-1}+2^{t-3}$, where $t\ge4$. Then:
    \[\zcl(W_n)\ge2^{t+1}+2^{t-3}-2.\]
\end{proposition}
\begin{proof}
We are going to prove $\zcl(W_{2^t+2^{t-1}+2})\ge2^{t+1}+2^{t-3}-2$, and then the proposition follows from Lemma \ref{zcl raste}. By applying Lemma \ref{z nonzero} for $\beta=2^{t+1}-1$, $\gamma=2^{t-3}-1$ and $r=2^{t+1}+2^{t-1}+1$ we see that it suffices to verify that
\[\sum_{2b+3c=2^{t+1}+2^{t-1}+1}
        \widetilde w_2^b\widetilde w_3^c\otimes
        \widetilde w_2^{2^{t+1}-1-b}\widetilde w_3^{2^{t-3}-1-c}\neq0 \quad\mbox{in }W_{2^t+2^{t-1}+2}\otimes W_{2^t+2^{t-1}+2}.\]
Similarly as in the proof of Proposition \ref{3/4 - donje1b - prvi}, we show that, excluding the summand with  $(b,c)=(2^{t}+2^{t-2}-1,1)$, all others are zero.

If $b>2^{t}+2^{t-2}-1$, then $\height(\widetilde w_2)=2^t+2^{t-2}-1$ (Theorem \ref{heights}) implies $\widetilde w_2^b\widetilde w_3^c=0$.

If $b<2^{t}+2^{t-2}-1$, then $c>1$. On the other hand, $c\le2^{t-3}-1$ implies $2b=2^{t+1}+2^{t-1}+1-3c\ge2^{t+1}+2^{t-3}+4$, i.e., $b\ge2^t+2^{t-4}+2>2^t-3$, and then Corollary \ref{treca cetvrtina - posledica} (along with (\ref{ekv - s tildom - bez tilde}) and the fact $I_{13\cdot2^{t-3}+1}\subseteq I_{2^t+2^{t-1}+2}$) gives us
\[\widetilde w_2^b\widetilde w_3^c=\widetilde w_2^{b-2^t+3}\widetilde w_2^{2^t-3}\widetilde w_3^c=\widetilde w_2^{b-2^t+3}\widetilde w_2^{2^{t-2}-3}\widetilde w_3^{2^{t-1}+c}=0,\]
since $\height(\widetilde w_3)=2^{t-1}+1$ (see Theorem \ref{heights}).

So, the above sum is equal to
    \begin{align*}
        \widetilde w_2^{2^t+2^{t-2}-1}\widetilde w_3\otimes\widetilde w_2^{2^{t-1}+2^{t-2}}\widetilde w_3^{2^{t-3}-2}.
    \end{align*}
    According to (\ref{cuplength1}), the first coordinate of this simple tensor is nonzero, and so we are left to prove $\widetilde w_2^{2^{t-1}+2^{t-2}}\widetilde w_3^{2^{t-3}-2}\neq0$ in $W_{2^t+2^{t-1}+2}$. However, in the proof of Proposition \ref{3/4 - donje1b - prvi} we established that the corresponding class is nonzero in $W_{2^t+2^{t-1}+1}$, which means (by (\ref{ekv - s tildom - bez tilde})) that $w_2^{2^{t-1}+2^{t-2}}w_3^{2^{t-3}-2}\notin I_{2^t+2^{t-1}+1}$. On the other hand, $I_{2^t+2^{t-1}+2}\subseteq I_{2^t+2^{t-1}+1}$, and so $w_2^{2^{t-1}+2^{t-2}}w_3^{2^{t-3}-2}\notin I_{2^t+2^{t-1}+2}$, i.e, $\widetilde w_2^{2^{t-1}+2^{t-2}}\widetilde w_3^{2^{t-3}-2}\neq0$ in $W_{2^t+2^{t-1}+2}$ as well.
\end{proof}

\subsection{The case $\mathbf{13\cdot2^{t-3}+1\le n\le2^t+2^{t-1}+2^{t-2}}$}

For the lower bound in this case we will use (\ref{cuplength0}) for $s=t-2$ and $n=2^t+2^{t-1}+2^{t-3}+1=13\cdot2^{t-3}+1$:
\begin{equation}\label{cuplength}
\widetilde w_2^{2^t+2^{t-2}-1}\widetilde w_3^{2^{t-3}}\neq0 \quad\mbox{in }W_{13\cdot2^{t-3}+1}.
\end{equation}

\begin{proposition}
Let $13\cdot2^{t-3}+1\le n\le2^t+2^{t-1}+2^{t-2}$, where $t\ge4$. Then
\[\zcl(W_n)\ge2^{t+1}+2^{t-2}-2.\]
\end{proposition}

\begin{proof} We know that $\zcl(W_n)$ increases with $n$ (Lemma \ref{zcl raste}), which means that it suffices to prove $\zcl(W_{13\cdot2^{t-3}+1})\ge2^{t+1}+2^{t-2}-2$.
The proof relies on Lemma \ref{z nonzero}. We apply that lemma for $\beta=2^{t+1}-1$, $\gamma=2^{t-2}-1$ and $r=23\cdot2^{t-3}-2$, and so we need to show that
\[\sum_{2b+3c=23\cdot2^{t-3}-2}\widetilde w_2^b\widetilde w_3^c\otimes\widetilde w_2^{2^{t+1}-1-b}\widetilde w_3^{2^{t-2}-1-c}\neq0 \quad\mbox{in }W_{13\cdot2^{t-3}+1}\otimes W_{13\cdot2^{t-3}+1}.\]
We are going to prove that the only nonzero summand in this sum is the one for $(b,c)=(2^t+2^{t-2}-1,2^{t-3})$. First of all, that summand is really nonzero, because it is $\widetilde w_2^{2^t+2^{t-2}-1}\widetilde w_3^{2^{t-3}}\otimes\widetilde w_2^{2^{t-1}+2^{t-2}}\widetilde w_3^{2^{t-3}-1}$, its first coordinate is nonzero by (\ref{cuplength}), and the second divides the first, so it is nonzero too. So, we are left to prove that all summands with $b\neq2^t+2^{t-2}-1$ are zero.

If $b>2^t+2^{t-2}-1$, then $\widetilde w_2^b\widetilde w_3^c=0$ since $\height(\widetilde w_2)=2^t+2^{t-2}-1$ (Theorem \ref{heights}).

If $2^t-3\le b<2^t+2^{t-2}-1$, then $c>2^{t-3}$ and by Corollary \ref{treca cetvrtina - posledica} one has
\[\widetilde w_2^b\widetilde w_3^c=\widetilde w_2^{b-2^t+3}\widetilde w_2^{2^t-3}\widetilde w_3^c=\widetilde w_2^{b-2^t+3}\widetilde w_2^{2^{t-2}-3}\widetilde w_3^{2^{t-1}+c}=0,\]
since $\height(\widetilde w_3)=2^{t-1}+2^{t-3}$ (see Theorem \ref{heights}).

Finally, there are no summands in the above sum with $b\le2^t-4$, since then $c\le2^{t-2}-1$ would imply $2b+3c\le2\cdot(2^t-4)+3\cdot(2^{t-2}-1)=22\cdot2^{t-3}-11<23\cdot2^{t-3}-2$.
\end{proof}

We are concluding the proof of Theorem \ref{zcl Wn} in this case by verifying the opposite inequality (the upper bound).

\begin{proposition}
Let $t\geq 4$ and $13\cdot2^{t-3}+1\le n\le2^t+2^{t-1}+2^{t-2}$. Then
\[\zcl(W_n)\le2^{t+1}+2^{t-2}-2.\]
\end{proposition}

\begin{proof}
As usual, we suppose to the contrary that there exist integers $\beta,\gamma\ge0$ such that
\[x:=z(\widetilde w_2)^\beta z(\widetilde w_3)^\gamma\neq0\quad\mbox{and}\quad\beta+\gamma=2^{t+1}+2^{t-2}-1.\]
Theorem \ref{heights} gives us $\height(\widetilde w_2)=2^t+2^{t-2}-1$ and $2^{t-1}+2^{t-3}\le\height(\widetilde w_3)\le2^{t-1}+2^{t-2}-1$, and then (\ref{htz}) implies $\height(z(\widetilde w_2))=2^{t+1}-1$ and $\height(z(\widetilde w_3))=2^t-1$. This means that $\beta\le2^{t+1}-1$ and $\gamma\le2^t-1$, and together with $\beta+\gamma=2^{t+1}+2^{t-2}-1$ these then imply $\gamma\ge2^{t-2}$ and $\beta\ge2^t+2^{t-2}$.

Note that the inequality $n-2\le2^t+2^{t-1}+2^{t-2}-2$ implies
\[w_3g_{2^t+2^{t-1}+2^{t-2}-3}=w_2g_{2^t+2^{t-1}+2^{t-2}-2}+g_{2^t+2^{t-1}+2^{t-2}}\in I_n,\]
which, according to Lemma \ref{g-3}(d), leads to
\begin{equation}\label{mon nnula}
\widetilde w_2^{2^{t-1}}\widetilde w_3^{2^{t-2}}=0 \quad\mbox{in } W_n.
\end{equation}
This is used for the second of the following equalities:
\begin{align*}
x&=z(\widetilde w_2)^{\beta-2^t}z(\widetilde w_3)^{\gamma-2^{t-2}}\big(\widetilde w_2^{2^t}\otimes 1+1\otimes\widetilde w_2^{2^t}\big)\big(\widetilde w_3^{2^{t-2}}\otimes 1+1\otimes\widetilde w_3^{2^{t-2}}\big)\\
&=z(\widetilde w_2)^{\beta-2^t}z(\widetilde w_3)^{\gamma-2^{t-2}}\big(\widetilde w_2^{2^t}\otimes\widetilde w_3^{2^{t-2}}+\widetilde w_3^{2^{t-2}}\otimes\widetilde w_2^{2^t}\big)\\
&=z(\widetilde w_3)^{\gamma-2^{t-2}}\sum_{i+j=\beta-2^t}\binom{\beta-2^t}{i}\widetilde w_2^i\otimes\widetilde w_2^j\big(\widetilde w_2^{2^t}\otimes\widetilde w_3^{2^{t-2}}+\widetilde w_3^{2^{t-2}}\otimes\widetilde w_2^{2^t}\big)\\
&=z(\widetilde w_3)^{\gamma-2^{t-2}}\sum_{i+j=\beta-2^t}\binom{\beta-2^t}{i}\big(\widetilde w_2^{2^t+i}\otimes\widetilde w_2^j\widetilde w_3^{2^{t-2}}+\widetilde w_2^i\widetilde w_3^{2^{t-2}}\otimes\widetilde w_2^{2^t+j}\big).
\end{align*}
Since $x\neq0$ there must exist nonnegative integers $i$ and $j$ with $i+j=\beta-2^t$ and $\widetilde w_2^{2^t+i}\otimes\widetilde w_2^j\widetilde w_3^{2^{t-2}}+\widetilde w_2^i\widetilde w_3^{2^{t-2}}\otimes\widetilde w_2^{2^t+j}\neq0$. Then $\height(\widetilde w_2)=2^t+2^{t-2}-1$ implies $\min\{i,j\}\le2^{t-2}-1$, and from (\ref{mon nnula}) we then conclude $\max\{i,j\}\le2^{t-1}-1$. This leads to $\beta-2^t=i+j\le2^{t-1}+2^{t-2}-2$, i.e., $\gamma\ge2^{t-1}+1$ (since $\beta+\gamma=2^{t+1}+2^{t-2}-1$).

Similarly as above, using (\ref{mon nnula}) we now obtain that
\[x=z(\widetilde w_2)^{\beta-2^t}z(\widetilde w_3)^{\gamma-2^{t-1}}\big(\widetilde w_2^{2^t}\otimes\widetilde w_3^{2^{t-1}}+\widetilde w_3^{2^{t-1}}\otimes\widetilde w_2^{2^t}\big).\]
Again by (\ref{mon nnula}) and the fact $\widetilde w_3^{2^{t-1}+2^{t-2}}=0$ (since $\height(\widetilde w_3)\le2^{t-1}+2^{t-2}-1$) we have that
\[z(\widetilde w_3)^{2^{t-2}}\big(\widetilde w_2^{2^t}\otimes\widetilde w_3^{2^{t-1}}\!+\widetilde w_3^{2^{t-1}}\otimes\widetilde w_2^{2^t}\big)\!=\!\big(\widetilde w_3^{2^{t-2}}\otimes1+1\otimes\widetilde w_3^{2^{t-2}}\big)\big(\widetilde w_2^{2^t}\otimes\widetilde w_3^{2^{t-1}}\!+\widetilde w_3^{2^{t-1}}\otimes\widetilde w_2^{2^t}\big)\]
is equal to zero, which means that we must have $\gamma-2^{t-1}\le2^{t-2}-1$ (since $x\neq0$). However, this implies $\beta=2^{t+1}+2^{t-2}-1-\gamma\ge2^t+2^{t-1}$, and so
\[x=z(\widetilde w_2)^{\beta-2^t-2^{t-1}}z(\widetilde w_3)^{\gamma-2^{t-1}}\big(\widetilde w_2^{2^t}\otimes\widetilde w_3^{2^{t-1}}+\widetilde w_3^{2^{t-1}}\otimes\widetilde w_2^{2^t}\big)\big(\widetilde w_2^{2^{t-1}}\otimes1+1\otimes\widetilde w_2^{2^{t-1}}\big)=0,\]
once again by (\ref{mon nnula}) and the fact $\height(\widetilde w_2)=2^t+2^{t-2}-1$. This contradicts the assumption $x\neq0$, and thus we are done.
\end{proof}

\subsection{The case $\mathbf{2^t+2^{t-1}+2^{t-2}+1\le n\le2^{t+1}-2}$}

Let $s\in\{1,2,\ldots,t-3\}$ be the (unique) integer such that $2^{t+1}-2^{s+1}+1\le n\le2^{t+1}-2^s$. We want to prove that $\zcl(W_n)=3\cdot2^t-2^{s+1}-2$, which will finish the proof of Theorem \ref{zcl Wn}.

For the lower bound we will need the Gr\"obner basis $F_{2^{t+1}-2^{s+1}+1}$.

\begin{lemma}\label{4/4 - Grebner - donje}
If $n=2^{t+1}-2^{s+1}+1$, then for $F_n=\{f_0,f_1,\ldots,f_{t-1}\}$ (the Gr\"obner basis from Theorem \ref{Grebner}) we have:
\begingroup
\allowbreak
\begin{itemize}
 \item $f_0=g_{2^{t+1}-2^{s+1}}$, $\mathrm{LM}(f_0)=w_2^{2^t-2^s}$;
 \item $f_i=w_3^{2^i-1}(g_{2^{t+1-i}-2^{s+1-i}-2})^{2^i}$, $\mathrm{LM}(f_i)=w_2^{2^t-2^s-2^i}w_3^{2^i-1}$, $1\le i\le s$;
 \item $f_i=w_3^{2^{i+1}-2^{s+1}+1}(g_{2^{t+1-i}-4})^{2^i}$, $\mathrm{LM}(f_i)=w_2^{2^t-2^{i+1}}w_3^{2^{i+1}-2^{s+1}+1}$,
        $s+1\le i\le t-1$.
\item $f_{t-2}=w_2^{2^{t-1}}w_3^{2^{t-1}-2^{s+1}+1}$, $f_{t-1}=w_3^{2^t-2^{s+1}+1}$.
\end{itemize}
\endgroup
\end{lemma}
\begin{proof}
We have $n-2^t+1=2^{t-1}+2^{t-2}+\dots+2^{s+1}+2$, and hence
$\alpha_0=0$, $\alpha_1=1$, $\alpha_2=\alpha_3=\dots=\alpha_s=0$, $\alpha_{s+1}=\alpha_{s+2}=\dots=\alpha_{t-1}=1$, $s_0=0$,
$s_i=2$ for $1\le i\le s$, and $s_i=2^i+2^{i-1}+\dots+2^{s+1}+2=2^{i+1}-2^{s+1}+2$ for
$s+1\le i\le t-1$. By definition, $f_i=w_3^{\alpha_is_{i-1}}g_{n-2+2^i-s_i}$, and so $f_0=g_{2^{t+1}-2^{s+1}}$,
\begin{align*}
f_i&=g_{2^{t+1}-2^{s+1}+2^i-3}=w_3^{2^i-1}(g_{2^{t+1-i}-2^{s+1-i}-2})^{2^i} \quad (\mbox{for }1\le i\le s),\\
f_i&=w_3^{2^i-2^{s+1}+2}g_{2^{t+1}-2^i-3}=w_3^{2^{i+1}-2^{s+1}+1}(g_{2^{t+1-i}-4})^{2^i} \quad (\mbox{for }s+1\le i\le t-1),
\end{align*}
by Lemma \ref{kvadriranje}. In particular, by looking at Table \ref{table:2} we see that
\begin{align*}
f_{t-2}&=w_3^{2^{t-1}-2^{s+1}+1}g_4^{2^{t-2}}=w_2^{2^{t-1}}w_3^{2^{t-1}-2^{s+1}+1},\\
f_{t-1}&=w_3^{2^t-2^{s+1}+1}g_0^{2^{t-1}}=w_3^{2^t-2^{s+1}+1}.
\end{align*}

The statements about leading monomials follow from $\mathrm{LM}(g_{2l})=w_2^l$ ($l\ge0$).
\end{proof}

In order to apply Lemma \ref{z nonzero} for cohomological dimension $r=2^{t+2}-3\cdot2^{s+1}-5$, we now detect all nonzero monomials of the form $\widetilde w_2^b\widetilde w_3^c$, where $2b+3c=2^{t+2}-3\cdot2^{s+1}-5$.

\begin{lemma}\label{ono sa s}
    Let $n=2^{t+1}-2^{s+1}+1$ ($1\le s\le t-3$). Then in $H^{2^{t+2}-3\cdot2^{s+1}-5}(\widetilde G_{n,3})$ the
    only nonzero monomials of the form $\widetilde w_2^b\widetilde w_3^c$ are
    $\widetilde w_2^{2^{t+1}-3\cdot2^{k-1}-1}\widetilde w_3^{2^k-2^{s+1}-1}$ for $s+2\le k\le t$.
    Furthermore, they are all equal, i.e.,
    \[\widetilde w_2^{2^{t+1}-3\cdot2^{k-1}-1}\widetilde w_3^{2^k-2^{s+1}-1}=\widetilde w_2^{2^{t-1}-1}\widetilde w_3^{2^t-2^{s+1}-1}\]
    for all $k\in\{s+2,s+3,\ldots,t\}$, and $\widetilde w_2^{2^{t-1}-1}\widetilde w_3^{2^t-2^{s+1}-1}\in\mathcal B_n$.
\end{lemma}
\begin{proof}
    It is obvious from Lemma \ref{4/4 - Grebner - donje} that $w_2^{2^{t-1}-1}w_3^{2^t-2^{s+1}-1}$ is not divisible by any of the leading monomials $\mathrm{LM}(f_i)$, $0\leq i\leq t-1$. So,
    $\widetilde w_2^{2^{t-1}-1}\widetilde w_3^{2^t-2^{s+1}-1}\in\mathcal B_n$, and particularly, $\widetilde w_2^{2^{t-1}-1}\widetilde w_3^{2^t-2^{s+1}-1}\neq0$.

    Note also that Theorem \ref{heights} gives us that $\height(\widetilde w_3)=2^t-2^{s+1}$, which means that $\widetilde w_2^b\widetilde w_3^c=0$ whenever $c>2^t-2^{s+1}$.

    \bigskip

    Now we use backward induction on $k$, where $s+2\le k\le t$, to prove that $\widetilde w_2^b\widetilde w_3^c$,
    with $2b+3c=2^{t+2}-3\cdot2^{s+1}-5$ and
    $2^k-2^{s+1}-1\leq c<2^{k+1}-2^{s+1}-1$, is nonzero if
    and only if $c=2^k-2^{s+1}-1$, and that $\widetilde w_2^{2^{t+1}-3\cdot2^{k-1}-1}\widetilde w_3^{2^k-2^{s+1}-1}=\widetilde w_2^{2^{t-1}-1}\widetilde w_3^{2^t-2^{s+1}-1}$. Throughout the proof we will use the fact that $c$ must be odd (this follows from
    $2b+3c=2^{t+2}-3\cdot2^{s+1}-5$).

    \medskip

    We have already discussed the cases $c=2^t-2^{s+1}-1$ and $c>2^t-2^{s+1}$. Since $c$ must be odd, the induction base ($k=t$) is verified.

    \medskip

    Proceeding to the induction step, let $s+2\le k\le t-1$ (and $2^k-2^{s+1}-1\le c<2^{k+1}-2^{s+1}-1$). We will distinguish two cases.
Suppose first that $c\ge2^k-2^{s+1}+1$. Note that $\mathrm{LM}(f_{k-1})\mid w_2^bw_3^c$ in this case. Namely, $\mathrm{LM}(f_{k-1})=w_2^{2^t-2^k}w_3^{2^k-2^{s+1}+1}$ (Lemma \ref{4/4 - Grebner - donje}), while
     \[c\ge2^k-2^{s+1}+1 \mbox{ and } b\ge\frac{1}{2}\left(2^{t+2}-3\cdot2^{s+1}-5-3\cdot(2^{k+1}-2^{s+1}-3)\right)>2^t-2^k.\]
So, we can reduce $w_2^bw_3^c$ by $f_{k-1}$, i.e., apply Lemma \ref{reduction - f_i} for $i=k-1$, and obtain
   \[\widetilde w_2^b\widetilde w_3^c=\sum_{\substack{2d+3e=2^{t+2-k}-4\\ e>0}}
              \binom{d+e}{e}\widetilde w_2^{b-2^{k-1}\left(2^{t+1-k}-2-d\right)}\widetilde w_3^{c+2^{k-1}e}.\]
    Obviously, in each summand of this sum $e$ must be even, and so $e>0$ implies $e\ge2$. Therefore, $c+2^{k-1}e\ge2^k-2^{s+1}+1+2^k>2^{k+1}-2^{s+1}-1$, which means that we can apply inductive hypothesis to conclude that if
    $\widetilde w_2^{b-2^t+2^k+2^{k-1}d}\widetilde w_3^{c+2^{k-1}e}\neq0$, then we must have
    $c+2^{k-1}e=2^m-2^{s+1}-1$ for some $m>k$. This implies $c\equiv-2^{s+1}-1\pmod{2^k}$,
    which is impossible since
    $2^k-2^{s+1}+1\le c\le2^{k+1}-2^{s+1}-3$ (and $s+1\le k-1$). So, each summand from the last sum is zero,
    and hence $\widetilde w_2^b\widetilde w_3^c=0$.

    Suppose now $c=2^k-2^{s+1}-1$. Then $b=2^{t+1}-3\cdot2^{k-1}-1>2^t-2^{s+1}$ (because $k\le t-1$) and $c>2^s-1$ (because $k\ge s+2$). Since $\mathrm{LM}(f_s)=w_2^{2^t-2^{s+1}}w_3^{2^s-1}$ (Lemma \ref{4/4 - Grebner - donje}) we can reduce $w_2^bw_3^c$ by $f_s$, and Lemma \ref{reduction - f_i} leads to
    \[\widetilde w_2^b\widetilde w_3^c=\sum_{\substack{2d+3e=2^{t+1-s}-4\\ e>0}}\binom{d+e}{e}\widetilde w_2^{b-2^s\left(2^{t-s}-2-d\right)}\widetilde w_3^{c+2^se}.\]
    Similarly as above, $e$ must be even, and so $e\ge2$. This means that $c+2^se\ge2^k-2^{s+1}-1+2^{s+1}=2^k-1>2^k-2^{s+1}+1$, and hence, by induction hypothesis and the first part of the induction step, if $\widetilde w_2^{b-2^t+2^{s+1}+2^sd}\widetilde w_3^{c+2^se}\neq0$, then $c+2^se=2^m-2^{s+1}-1$ for some $m\in\{k+1,k+2,\ldots,t\}$. This implies $e=2^{m-s}-2^{k-s}$ and $d=(2^{t+1-s}-4-3\cdot2^{m-s}+3\cdot2^{k-s})/2=2^{t-s}-2-3\cdot2^{m-s-1}+3\cdot2^{k-s-1}$. Therefore,
    \[\widetilde w_2^b\widetilde w_3^c=\sum_{m=k+1}^t\binom{2^{t-s}-2-2^{m-s-1}+2^{k-s-1}}{2^{m-s}-2^{k-s}}\widetilde w_2^{2^{t+1}-3\cdot2^{m-1}-1}\widetilde w_3^{2^m-2^{s+1}-1}.\]
    For $m=k+1$ the binomial coefficient is $\binom{2^{t-s}-2-2^{k-s-1}}{2^{k-s}}\equiv1\pmod2$ by Lucas' theorem, since $2^{t-s}-2-2^{k-s-1}\equiv2^{k-s}+2^{k-s-1}-2\pmod{2^{k-s+1}}$. On the other hand, for $k+2\le m\le t$ the binomial coefficient vanishes, since $2^{t-s}-2-2^{m-s-1}+2^{k-s-1}\equiv2^{k-s-1}-2\pmod{2^{k-s+1}}$, while $2^{m-s}-2^{k-s}\equiv2^{k-s}\pmod{2^{k-s+1}}$. So the only nonzero summand in the above sum is the one for $m=k+1$. Finally, we conclude that
    \[\widetilde w_2^{2^{t+1}-3\cdot2^{k-1}\!-1}\widetilde w_3^{2^k-2^{s+1}\!-1}\!\!=\widetilde w_2^b\widetilde w_3^c=\widetilde w_2^{2^{t+1}-3\cdot2^k\!-1}\widetilde w_3^{2^{k+1}-2^{s+1}\!-1}\!\!=\widetilde w_2^{2^{t-1}\!-1}\widetilde w_3^{2^t-2^{s+1}\!-1},\]
    by induction hypothesis.

\bigskip

We are left to prove that $\widetilde w_2^b\widetilde w_3^c=0$ if $c<2^{s+2}-2^{s+1}-1=2^{s+1}-1$ (and $2b+3c=2^{t+2}-3\cdot2^{s+1}-5$). Actually, we are going to continue the above backward induction on $k$, and prove that for $1\le k\le s$, $2^k-1\le c<2^{k+1}-1$ (and $2b+3c=2^{t+2}-3\cdot2^{s+1}-5$) one has $\widetilde w_2^b\widetilde w_3^c=0$.

So, let $2^k-1\le c<2^{k+1}-1$ for some $k\in\{1,2,\ldots,s\}$. Then
    $b>(2^{t+2}-3\cdot2^{s+1}-5-3\cdot2^{k+1}+3)/2=2^{t+1}-3\cdot2^s-3\cdot2^k-1>2^t-2^s-2^k$ (since $k\le s\le t-3$). By Lemma \ref{4/4 - Grebner - donje} we know that $\mathrm{LM}(f_k)=w_2^{2^t-2^s-2^k}w_3^{2^k-1}$, and therefore we can
    reduce $w_2^bw_3^c$ by $f_k$. Lemma \ref{reduction - f_i} then gives us
    \[\widetilde w_2^b\widetilde w_3^c=\sum_{\substack{2d+3e=2^{t+1-k}-2^{s+1-k}-2\\ e>0}}\binom{d+e}{e}\widetilde w_2^{b-2^k\left(2^{t-k}-2^{s-k}-1-d\right)}\widetilde w_3^{c+2^ke}.\]
Again, $e$ must be even, which leads to $e\ge2$, and consequently $c+2^ke>2^ke\ge2^{k+1}$. By inductive hypothesis, in order for $\widetilde w_2^{b-2^t+2^s+2^k+2^kd}\widetilde w_3^{c+2^ke}$ to be nonzero, we must have $c+2^ke=2^m-2^{s+1}-1$ for some $m\in\{s+2,s+3,\ldots,t\}$. However, this would imply $c\equiv-1\pmod{2^{k+1}}$, which contradicts the fact $2^k-1\le c<2^{k+1}-1$. Finally, we conclude $\widetilde w_2^b\widetilde w_3^c=0$, and the proof of the lemma is now complete.
\end{proof}

We are now ready to establish the lower bound.

\begin{proposition}
    Let $2^{t+1}-2^{s+1}+1\le n \le2^{t+1}-2^s$ (where $1\le s\le t-3$). Then:
    \[\zcl(W_{n})\ge3\cdot2^t-2^{s+1}-2.\]
\end{proposition}
\begin{proof}
    By Lemma \ref{zcl raste}, it is enough to prove the inequality for
    $n=2^{t+1}-2^{s+1}+1$. In order to do so, we apply Lemma \ref{z nonzero} for $\beta=2^{t+1}-1$, $\gamma=2^t-2^{s+1}-1$ and $r=2^{t+2}-3\cdot2^{s+1}-5$. By that lemma, it suffices to show that
    \[\sum_{2b+3c=2^{t+2}-3\cdot 2^{s+1}-5}\!
        \binom{2^{t+1}-1}{b}\binom{2^t-2^{s+1}-1}{c}
        \widetilde w_2^b\widetilde w_3^c\otimes
        \widetilde w_2^{2^{t+1}-1-b}\widetilde w_3^{2^t-2^{s+1}-1-c}\!\neq0\]
    in $W_n\otimes W_n$. By Lemma \ref{ono sa s}, we only need to consider the summands with
    $(b,c)=(2^{t+1}-3\cdot 2^{k-1}-1,2^k-2^{s+1}-1)$, for $s+2\le k\le t$, and for each
    of them we know that $\widetilde w_2^b\widetilde w_3^c=\widetilde w_2^{2^{t-1}-1}\widetilde w_3^{2^t-2^{s+1}-1}$. So, by Lucas' theorem,
    the last sum simplifies to
    \[\widetilde w_2^{2^{t-1}-1}\widetilde w_3^{2^t-2^{s+1}-1}\otimes\sum_{k=s+2}^t\widetilde w_2^{3\cdot 2^{k-1}}\widetilde w_3^{2^t-2^k}.\]
    We know that $\widetilde w_2^{2^{t-1}-1}\widetilde w_3^{2^t-2^{s+1}-1}\neq0$ (by Lemma \ref{ono sa s}), so it remains to establish that $\sum_{k=s+2}^t\widetilde w_2^{3\cdot 2^{k-1}}\widetilde w_3^{2^t-2^k}\neq0$ in $W_n$. However, we are going to prove that this sum multiplied by $\widetilde w_2^{2^{t-1}-3\cdot2^s-1}$ is nonzero, so the sum itself must be nonzero. Namely:
    \begin{align*}
    \widetilde w_2^{2^{t-1}-3\cdot2^s-1}\sum_{k=s+2}^t\widetilde w_2^{3\cdot 2^{k-1}}\widetilde w_3^{2^t-2^k}&=\sum_{k=s+2}^t\widetilde w_2^{2^{t-1}-3\cdot2^s-1+3\cdot2^{k-1}}\widetilde w_3^{2^t-2^k}\\
    &=\widetilde w_2^{2^{t+1}-3\cdot2^s-1}\neq0.
    \end{align*}
    The second equality holds because for $s+2\le k<t$ one has $2^{t-1}-3\cdot2^s-1+3\cdot2^{k-1}>2^{t-1}$ and $2^t-2^k>2^{t-1}-2^{s+1}+1$, and then $f_{t-2}=w_2^{2^{t-1}}w_3^{2^{t-1}-2^{s+1}+1}$ (Lemma \ref{4/4 - Grebner - donje}) implies that $\widetilde w_2^{2^{t-1}-3\cdot2^s-1+3\cdot2^{k-1}}\widetilde w_3^{2^t-2^k}=0$ in this case. So, only the summand for $k=t$ remains, and $\widetilde w_2^{2^{t+1}-3\cdot2^s-1}\neq0$ because $\height(\widetilde w_2)$ is exactly $2^{t+1}-3\cdot2^s-1$ (see Theorem \ref{heights}).
\end{proof}

In the final part of the section we obtain the upper bound for $\zcl(W_n)$ in the case
$2^{t+1}-2^{s+1}+1\leq n\leq 2^{t+1}-2^s$ (where $t\ge4$ and $1\le s\le t-3$).
For this proof we will need the following two lemmas.

\begin{lemma}\label{4/4 - Grebner}
In $W_{2^{t+1}-2^s}$ the following relations hold:
\begin{itemize}
    \item[(a)] $\widetilde w_2^{2^{t-1}+2^{t-2}}\widetilde w_3^{2^{t-2}-2^s}=\widetilde w_3^{2^{t-1}+2^{t-2}-2^s}$;
    \item[(b)] $\widetilde w_2^{2^{t-1}}\widetilde w_3^{2^{t-1}-2^s}=0$;
    \item[(c)] $\widetilde w_3^{2^t-2^s}=0$.
\end{itemize}
\end{lemma}

\begin{proof}
We will use the members of the Gr\"obner basis $F_{2^{t+1}-2^s}$, defined in Theorem \ref{Grebner}.
If $n=2^{t+1}-2^s$, then $n-2^t+1=2^{t-1}+2^{t-2}+\dots+2^s+1$, and since $s\le t-3$, we have $\alpha_{t-3}=\alpha_{t-2}=\alpha_{t-1}=1$ and $s_i=2^{i+1}-2^s+1$ for $t-4\le i\le t-1$.
Therefore, for $t-3\le i\le t-1$, by Lemma \ref{kvadriranje} we have
\[f_i=w_3^{\alpha_is_{i-1}}g_{n-2+2^i-s_i}=w_3^{2^i-2^s+1}g_{2^{t+1}-2^i-3}=w_3^{2^{i+1}-2^s}(g_{2^{t-i+1}-4})^{2^i}.\]
Now Table \ref{table:2} leads to
\begin{align*}
f_{t-3}&=w_3^{2^{t-2}-2^s}g_{12}^{2^{t-3}}=w_3^{2^{t-2}-2^s}(w_2^6+w_3^4)^{2^{t-3}}\!=w_2^{3\cdot2^{t-2}}w_3^{2^{t-2}-2^s}+w_3^{2^{t-1}+2^{t-2}-2^s};\\
f_{t-2}&=w_3^{2^{t-1}-2^s}g_4^{2^{t-2}}=w_2^{2^{t-1}}w_3^{2^{t-1}-2^s};\\
f_{t-1}&=w_3^{2^t-2^s}g_0^{2^{t-1}}=w_3^{2^t-2^s}.
\end{align*}
Of course, all these polynomials belong to $I_{2^{t+1}-2^s}$, and so (\ref{ekv - s tildom - bez tilde}) concludes the proof of the lemma.
\end{proof}

\begin{lemma}\label{4/4 - z}
Let $n=2^{t+1}-2^s$ (where $t\ge4$ and $1\le s\le t-3$).
Then in $W_n\otimes W_n$ one has:
\[z(\widetilde w_2)^{2^{t+1}-2^{s+1}}z(\widetilde w_3)^{2^t-2^s}=0 \quad\mbox{ and }\quad z(\widetilde w_2)^{2^{t+1}-2^s}z(\widetilde w_3)^{2^t-2^{s+1}}=0.\]
\end{lemma}

\begin{proof}
Let $x:=z(\widetilde w_2)^{2^{t+1}-2^{s+1}}z(\widetilde w_3)^{2^t-2^{s+1}}$. We need to prove that $x\cdot z(\widetilde w_3)^{2^s}=0$ and $x\cdot z(\widetilde w_2)^{2^s}=0$. Let us  first establish that
\begin{equation}\label{x=sum}
x=A^{2^s}\cdot(\widetilde w_3\otimes\widetilde w_3)^{2^{t}-2^{s+1}},
\end{equation}
where
\[A=\sum_{i=0}^{2^{t-s-2}-2}\bigg(\widetilde w_2^{3+2i}\otimes\widetilde w_2^{2^{t-s-1}-2-2i}+\widetilde w_2^{2^{t-s-1}-2-2i}\otimes\widetilde w_2^{3+2i}\bigg).\]

According to (\ref{z(a^stepen dvojke)}) and using the fact $s+1\le t-2$, we have
\[x=z(\widetilde w_2)^{2^{t-1}-2^{s+1}}z(\widetilde w_3)^{2^{t-2}-2^{s+1}}z\big(\widetilde w_2^{2^{t}}\big)z\big(\widetilde w_2^{2^{t-1}}\big)z\big(\widetilde w_3^{2^{t-1}}\big)z\big(\widetilde w_3^{2^{t-2}}\big).\]
By Lemma \ref{4/4 - Grebner}(b), $\widetilde w_2^{2^{t-1}}\widetilde w_3^{2^{t-1}}=0$, and hence
\[
x=z(\widetilde w_2)^{2^{t-1}-2^{s+1}}z(\widetilde w_3)^{2^{t-2}-2^{s+1}}z\big(\widetilde w_3^{2^{t-2}}\big)\big(\widetilde w_2^{2^{t}+2^{t-1}}\otimes\widetilde w_3^{2^{t-1}}+\widetilde w_3^{2^{t-1}}\otimes\widetilde w_2^{2^t+2^{t-1}}\big).
\]
Again by Lemma \ref{4/4 - Grebner} (part (a) multiplied by $\widetilde w_2^{2^{t-1}+2^{t-2}}\widetilde w_3^{2^s}$ and part (b) multiplied by $\widetilde w_2^{2^{t-2}}\widetilde w_3^{2^{t-2}+2^s}$), $\widetilde w_2^{2^t+2^{t-1}}\widetilde w_3^{2^{t-2}}=\widetilde w_2^{2^{t-1}+2^{t-2}}\widetilde w_3^{2^{t-1}+2^{t-2}}=0$, and so
\[
x=z(\widetilde w_2)^{2^{t-1}-2^{s+1}}z(\widetilde w_3)^{2^{t-2}-2^{s+1}}\big(\widetilde w_2^{2^t+2^{t-1}}\otimes\widetilde w_3^{2^{t-1}+2^{t-2}}+\widetilde w_3^{2^{t-1}+2^{t-2}}\otimes\widetilde w_2^{2^t+2^{t-1}}\big).
\]
As in some instances before, we will expand the first two factors by binomial formula, while the expression in the brackets will be expanded using Corollary \ref{2/4 - posledica}. By that corollary (applied for $t+1$ in place of $t$), $w_2^{3\cdot 2^{t-1}}\equiv\sum_{k=1}^{t-2}w_2^{3\cdot 2^{k-1}}w_3^{2^t-2^k}\pmod{I_{2^{t+1}+2^{t-1}+2}}$, and since $I_n=I_{2^{t+1}-2^s}\supseteq I_{2^{t+1}+2^{t-1}+2}$, using (\ref{ekv - s tildom - bez tilde}) we get
\[\widetilde w_2^{2^t+2^{t-1}}=\sum_{k=1}^{t-2}\widetilde w_2^{3\cdot 2^{k-1}}\widetilde w_3^{2^t-2^k} \quad\mbox{in }W_n.\]
Every summand from the expansion of $z(\widetilde w_2)^{2^{t-1}-2^{s+1}}z(\widetilde w_3)^{2^{t-2}-2^{s+1}}$ is of the
form
\[\binom{2^{t-1}-2^{s+1}}{l}\binom{2^{t-2}-2^{s+1}}{m}\widetilde w_2^{l}\widetilde w_3^{m}\otimes\widetilde w_2^{2^{t-1}-2^{s+1}-l}\widetilde w_3^{2^{t-2}-2^{s+1}-m},\]
where $l$ and $m$ are integers such that $0\le l\le2^{t-1}-2^{s+1}$ and $0\le m\le2^{t-2}-2^{s+1}$.
By Lucas' theorem, $\binom{2^{t-1}-2^{s+1}}{l}$ (resp.\
$\binom{2^{t-2}-2^{s+1}}{m}$) is odd
if and only if $l=2^{s+1}i$ for some $i\in\{0,1,\ldots,2^{t-s-2}-1\}$ (resp.\
$m=2^{s+1}j$ for some $j\in\{0,1,\ldots,2^{t-s-3}-1\}$). We conclude that
\begin{align*}
  x & =\sum_{i=0}^{2^{t-s-2}-1}\sum_{j=0}^{2^{t-s-3}-1}\sum_{k=1}^{t-2}\widetilde w_2^{3\cdot2^{k-1}+2^{s+1}i}\widetilde w_3^{2^t-2^k+2^{s+1}j}\otimes \widetilde w_2^{2^{t-1}-2^{s+1}(i+1)}\widetilde w_3^{2^t-2^{s+1}(j+1)}\\
   & +\sum_{i=0}^{2^{t-s-2}-1}\sum_{j=0}^{2^{t-s-3}-1}\sum_{k=1}^{t-2}\widetilde w_2^{2^{t-1}-2^{s+1}(i+1)}\widetilde w_3^{2^t-2^{s+1}(j+1)}\otimes \widetilde w_2^{3\cdot2^{k-1}+2^{s+1}i}\widetilde w_3^{2^t-2^k+2^{s+1}j}.
\end{align*}
Also, Lemma \ref{4/4 - Grebner}(c) implies $\widetilde w_3^{2^t-2^k+2^{s+1}j}=0$ for
$k\le s$ (and all $j\ge0$), while for $k=s+1$ the class $\widetilde w_3^{2^t-2^k+2^{s+1}j}$ is nonzero only for
$j=0$. Moreover, in this case,
i.e., $j=0$ and $k=s+1$, for
$i=2^{t-s-2}-1$ one has
$\widetilde w_2^{3\cdot2^{k-1}+2^{s+1}i}\widetilde w_3^{2^t-2^k+2^{s+1}j}=\widetilde w_2^{2^{t-1}+2^s}\widetilde w_3^{2^t-2^{s+1}}=0$
(by Lemma \ref{4/4 - Grebner}(b)). Hence, $x$ is equal to
\begin{align*}
 & \sum_{k=s+2}^{t-2}\sum_{i=0}^{2^{t-s-2}-1}\,\sum_{j=0}^{2^{t-s-3}-1}\widetilde w_2^{3\cdot2^{k-1}+2^{s+1}i}\widetilde w_3^{2^t-2^k+2^{s+1}j}\otimes\widetilde w_2^{2^{t-1}-2^{s+1}(i+1)}\widetilde w_3^{2^t-2^{s+1}(j+1)}\\
   & +\sum_{k=s+2}^{t-2}\sum_{i=0}^{2^{t-s-2}-1}\,\sum_{j=0}^{2^{t-s-3}-1}\!\widetilde w_2^{2^{t-1}-2^{s+1}(i+1)}\widetilde w_3^{2^t-2^{s+1}(j+1)}\!\otimes\widetilde w_2^{3\cdot2^{k-1}+2^{s+1}i}\widetilde w_3^{2^t-2^k+2^{s+1}j}\\
   &+y\cdot\sum_{i=0}^{2^{t-s-2}-2}\bigg(\widetilde w_2^{3\cdot2^s+2^{s+1}i}\otimes \widetilde w_2^{2^{t-1}-2^{s+1}(i+1)}+\widetilde w_2^{2^{t-1}-2^{s+1}(i+1)}\otimes \widetilde w_2^{3\cdot2^s+2^{s+1}i}\bigg),
\end{align*}
where $y=\widetilde w_3^{2^{t}-2^{s+1}}\otimes \widetilde w_3^{2^{t}-2^{s+1}}=(\widetilde w_3\otimes \widetilde w_3)^{2^{t}-2^{s+1}}$. It is obvious that the last sum is $A^{2^s}$ (see (\ref{x=sum})), and if we denote the two triple sums by $x_1$ and $x_2$ respectively, we will prove (\ref{x=sum}) as soon as we show that $x_1=x_2$ (i.e., that these two triple sums cancel out).

Let us introduce the following notation:
\[\sigma(k,i,j)\!:=\!\widetilde w_2^{3\cdot2^{k-1}+2^{s+1}i}\widetilde w_3^{2^t-2^k+2^{s+1}j} \mbox{ and } \tau(k,i,j)\!:=\!\widetilde w_2^{2^{t-1}-2^{s+1}(i+1)}\widetilde w_3^{2^t-2^{s+1}(j+1)}.\]
We know that $\widetilde w_2^{2^{t-1}}\widetilde w_3^{2^{t-1}-2^s}=0$ (Lemma \ref{4/4 - Grebner}(b)), and since $2^t-2^k+2^{s+1}j\ge2^t-2^{t-2}>2^{t-1}-2^s$, we conclude that we must have $3\cdot2^{k-1}+2^{s+1}i<2^{t-1}$, i.e., $i\le2^{t-s-2}-3\cdot2^{k-s-2}-1$, in order for $\sigma(k,i,j)$ to be nonzero. Similarly, since $\widetilde w_3^{2^t-2^s}=0$ (Lemma \ref{4/4 - Grebner}(c)), one more necessary condition for $\sigma(k,i,j)\neq0$ is $2^t-2^k+2^{s+1}j<2^t-2^s$, and this amounts to $2^t-2^k+2^{s+1}j\le2^t-2^{s+1}$ because $2^t-2^k+2^{s+1}j$ is divisible by $2^{s+1}$. So, $\sigma(k,i,j)\neq0$ implies $j\le2^{k-s-1}-1$. We conclude that
\begin{align*}
x_1&=\sum_{k=s+2}^{t-2}\,\sum_{i=0}^{2^{t-s-2}-3\cdot2^{k-s-2}-1}\,\,\sum_{j=0}^{2^{k-s-1}-1}\sigma(k,i,j)\otimes\tau(k,i,j)\quad \mbox{and}\\
x_2&=\sum_{k=s+2}^{t-2}\,\sum_{\overline i=0}^{2^{t-s-2}-3\cdot2^{k-s-2}-1}\,\,\sum_{\overline j=0}^{2^{k-s-1}-1}\tau(k,\overline i,\overline j)\otimes\sigma(k,\overline i,\overline j).
\end{align*}
However, if $i+\overline i=2^{t-s-2}-3\cdot2^{k-s-2}-1$ and $j+\overline j=2^{k-s-1}-1$, then it is routine to check that
\[\sigma(k,i,j)=\tau(k,\overline i,\overline j) \quad\mbox{and}\quad \tau(k,i,j)=\sigma(k,\overline i,\overline j),\]
so the change of variables $\overline i=2^{t-s-2}-3\cdot2^{k-s-2}-1-i$ and $\overline j=2^{k-s-1}-1-j$ transforms the sum $x_1$ to the sum $x_2$, leading to the conclusion $x_1=x_2$. This establishes (\ref{x=sum}).

\medskip

Now we prove $x\cdot z(\widetilde w_3)^{2^s}=0$. By (\ref{z(a^stepen dvojke)}) and Lemma \ref{4/4 - Grebner}(c),
\begin{align*}
(\widetilde w_3\otimes\widetilde w_3)^{2^{t}-2^{s+1}}z(\widetilde w_3)^{2^s}&=\big(\widetilde w_3^{2^{t}-2^{s+1}}\otimes \widetilde w_3^{2^{t}-2^{s+1}}\big)\big(\widetilde w_3^{2^s}\otimes1+1\otimes\widetilde w_3^{2^s}\big)\\
&=\widetilde w_3^{2^{t}-2^s}\otimes \widetilde w_3^{2^{t}-2^{s+1}}+\widetilde w_3^{2^{t}-2^{s+1}}\otimes \widetilde w_3^{2^{t}-2^s}=0,
\end{align*}
and then (\ref{x=sum}) implies $x\cdot z(\widetilde w_3)^{2^s}=0$.

\medskip

To prove $x\cdot z(\widetilde w_2)^{2^s}=0$ we also use (\ref{x=sum}). Let us first calculate $z(\widetilde w_2)\cdot A$. This element is equal to
\begin{align*}
&(\widetilde w_2\otimes1+1\otimes\widetilde w_2)\cdot\sum_{i=0}^{2^{t-s-2}-2}\bigg(\widetilde w_2^{3+2i}\otimes\widetilde w_2^{2^{t-s-1}-2-2i}+\widetilde w_2^{2^{t-s-1}-2-2i}\otimes\widetilde w_2^{3+2i}\bigg)\\
=&\sum_{i=0}^{2^{t-s-2}-2}\widetilde w_2^{4+2i}\otimes\widetilde w_2^{2^{t-s-1}-2-2i}+\sum_{i=0}^{2^{t-s-2}-2}\widetilde w_2^{2^{t-s-1}-1-2i}\otimes\widetilde w_2^{3+2i}\\
&+\sum_{i=0}^{2^{t-s-2}-2}\widetilde w_2^{3+2i}\otimes\widetilde w_2^{2^{t-s-1}-1-2i}+\sum_{i=0}^{2^{t-s-2}-2}\widetilde w_2^{2^{t-s-1}-2-2i}\otimes\widetilde w_2^{4+2i}.
\end{align*}
The change of variable $i\mapsto2^{t-s-2}-2-i$ transforms the second sum to the third, so these two cancel out. For the first and the fourth, if we apply the change of variable $i\mapsto2^{t-s-2}-3-i$ to the fourth sum, we get
\begin{align*}
z(\widetilde w_2)\cdot A&=\sum_{i=0}^{2^{t-s-2}-2}\widetilde w_2^{4+2i}\otimes\widetilde w_2^{2^{t-s-1}-2-2i}+\sum_{i=-1}^{2^{t-s-2}-3}\widetilde w_2^{4+2i}\otimes\widetilde w_2^{2^{t-s-1}-2-2i}\\
&=\widetilde w_2^{2^{t-s-1}}\otimes\widetilde w_2^2+\widetilde w_2^2\otimes\widetilde w_2^{2^{t-s-1}}.
\end{align*}
Now, by (\ref{x=sum}) we have
\begin{align*}
z(\widetilde w_2)^{2^s}\cdot x&=\big(z(\widetilde w_2)\cdot A\big)^{2^s}(\widetilde w_3\otimes\widetilde w_3)^{2^{t}-2^{s+1}}\\
&=\big(\widetilde w_2^{2^{t-1}}\otimes\widetilde w_2^{2^{s+1}}+\widetilde w_2^{2^{s+1}}\otimes\widetilde w_2^{2^{t-1}}\big)\big(\widetilde w_3^{2^{t}-2^{s+1}}\otimes\widetilde w_3^{2^{t}-2^{s+1}}\big)\\
&=\widetilde w_2^{2^{t-1}}\widetilde w_3^{2^{t}-2^{s+1}}\otimes\widetilde w_2^{2^{s+1}}\widetilde w_3^{2^{t}-2^{s+1}}+\widetilde w_2^{2^{s+1}}\widetilde w_3^{2^{t}-2^{s+1}}\otimes\widetilde w_2^{2^{t-1}}\widetilde w_3^{2^{t}-2^{s+1}}\\
&=0,
\end{align*}
by Lemma \ref{4/4 - Grebner}(b), and the proof is complete.
\end{proof}

We can now easily prove the upper bound in this case.

\begin{proposition}
Let $1\le s\le t-3$ and
$2^{t+1}-2^{s+1}+1\leq n\leq 2^{t+1}-2^s$. Then
\[\zcl(W_n)\le3\cdot2^t-2^{s+1}-2.\]
\end{proposition}

\begin{proof}
By Lemma \ref{zcl raste}, $\zcl(W_n)\le\zcl(W_{2^{t+1}-2^s})$, so it is enough to prove that $\zcl(W_{2^{t+1}-2^s})\le3\cdot2^t-2^{s+1}-2$.
Assume to the contrary that there are integers $\beta,\gamma\ge0$ such that $\beta+\gamma=2^{t+1}+2^t-2^{s+1}-1$ and
\[z(\widetilde w_2)^\beta z(\widetilde w_3)^\gamma\neq 0\quad\mbox{in } W_{2^{t+1}-2^s}\otimes W_{2^{t+1}-2^s}.\]
We know that $\height(\widetilde w_2)=2^{t+1}-3\cdot2^s-1$ and
$\height(\widetilde w_3)=2^t-2^s-1$ (by Theorem \ref{heights}), and then (\ref{htz})
gives us $\height(z(\widetilde w_2))=2^{t+1}-1$
and $\height(z(\widetilde w_3))=2^t-1$. This means that we must have $\beta\le2^{t+1}-1$ and $\gamma\le2^t-1$, and consequently $\beta=2^{t+1}+2^t-2^{s+1}-1-\gamma\ge2^{t+1}-2^{s+1}$ and $\gamma\ge2^t-2^{s+1}$. Now for $\beta':=\beta-2^{t+1}+2^{s+1}$ and
$\gamma':=\gamma-2^t+2^{s+1}$ one has $\beta'+\gamma'=2^{s+1}-1$, and hence either $\beta'\ge2^s$ or $\gamma'\ge2^s$, i.e., either $\beta\ge2^{t+1}-2^s$ or $\gamma\ge2^t-2^s$. However, this contradicts Lemma \ref{4/4 - z}.
\end{proof}

The proof of Theorem \ref{zcl Wn} is now completed.

\section{Comparison between $\zcl(\widetilde G_{n,3})$ and $\zcl(W_n)$}
\label{comp}

In this section we compare $\zcl(W_n)$ with $\zcl(\widetilde G_{n,3})$ and thus establish some lower bounds for $\mathrm{TC}(\widetilde G_{n,3})$. Since we are working over a field, $W_n\subseteq H^*(\widetilde G_{n,3})$ implies $W_n\otimes W_n\subseteq H^*(\widetilde G_{n,3})\otimes H^*(\widetilde G_{n,3})$, and so we certainly have $\zcl(\widetilde G_{n,3})\ge\zcl(W_n)$. However, we can prove more.

\begin{proposition}
 For every integer $n\ge6$ one has
 \[\zcl(\widetilde G_{n,3})\ge1+\zcl(W_n).\]
\end{proposition}
\begin{proof}
	Let $z=z(\widetilde w_2)^\beta z(\widetilde w_3)^\gamma$ be a monomial which realizes $\zcl(W_n)$ (so $\zcl(W_n)=\beta+\gamma$), and let $m=2\beta+3\gamma$. Since
	\[z\neq0 \quad\mbox{ in }\big(H^*(\widetilde G_{n,3})\otimes H^*(\widetilde G_{n,3})\big)_m=\bigoplus_{i=0}^mH^i(\widetilde G_{n,3})\otimes H^{m-i}(\widetilde G_{n,3}),\]
	if we write $z$ in the form $\sum_{i=0}^mz_i$, where $z_i\in H^i(\widetilde G_{n,3})\otimes H^{m-i}(\widetilde G_{n,3})$, then there exists $k\in\{0,1,\ldots,m\}$ with the property $z_k\neq0$. 

Now, let $\{b_1,b_2,\ldots,b_r\}$ be a vector space basis for  $H^k(\widetilde G_{n,3})\cap W_n$, and complete it to a basis $\{b_1,\ldots,b_r,b_{r+1},\ldots,b_s\}$ for $H^k(\widetilde G_{n,3})$. Then $z_k$ can be written in the form $\sum_{j=1}^rb_j\otimes v_j$ for some $v_j\in H^{m-k}(\widetilde G_{n,3})\cap W_n$.
	Since $z_k\neq0$, there exists $j_0\in\{1,2,\ldots,r\}$ with the property $b_{j_0}\otimes v_{j_0}\neq0$. Note that (\ref{topdim}) implies $k<3n-9$, because $b_{j_0}\in H^k(\widetilde G_{n,3})$ is a nonzero polynomial in $\widetilde w_2$ and $\widetilde w_3$.
	
	Define now a map $\varphi:H^k(\widetilde G_{n,3})\rightarrow H^{3n-9}(\widetilde G_{n,3})$ on the basis elements by requiring $\varphi(b_{j_0})=c$, where $c\in H^{3n-9}(\widetilde G_{n,3})$ is the generator, and $\varphi(b_{j})=0$ for $j\neq j_0$. Poincar\'e duality applies to give us a class $a\in H^{3n-9-k}(\widetilde G_{n,3})$ such that $\varphi$ is multiplication with $a$. So for $1\le j\le s$, one has $ab_j\neq0$ if and only if $j=j_0$.
	
	Let us show that $z(a)z(\widetilde w_2)^\beta z(\widetilde w_3)^\gamma\neq0$ in $H^*(\widetilde G_{n,3})\otimes H^*(\widetilde G_{n,3})$, which will prove the proposition. The degree of this element is $3n-9-k+m$, and we have
	\[z(a)z(\widetilde w_2)^\beta z(\widetilde w_3)^\gamma=(a\otimes1+1\otimes a)\sum_{i=0}^mz_i=\sum_{i=0}^m(a\otimes1)z_i+\sum_{i=0}^m(1\otimes a)z_i.\]
	The summand in $H^{3n-9}(\widetilde G_{n,3})\otimes H^{m-k}(\widetilde G_{n,3})$ is either $(a\otimes1)z_k+(1\otimes a)z_{3n-9}$ (if $m\ge3n-9$) or $(a\otimes1)z_k$ (if $m<3n-9$). In the former case, we have $z_{3n-9}=0$, since the first coordinates of the simple tensors in $z_{3n-9}$ belong to $H^{3n-9}(\widetilde G_{n,3})\cap W_n=0$ (by (\ref{topdim})). So in any case, this summand is $(a\otimes1)z_k$ and it suffices to prove that it is nonzero. This is a consequence of the choice of the class $a$:
	\[(a\otimes1)z_k=(a\otimes1)\sum_{j=1}^rb_j\otimes v_j=\sum_{j=1}^rab_j\otimes v_j=ab_{j_0}\otimes v_{j_0}\neq0.\]
\end{proof}

So the difference between $\zcl(\widetilde G_{n,3})$ and $\zcl(W_n)$ is at least $1$, but, as we show in the next proposition, not more than $2$.

\begin{proposition}\label{c2}
 For every integer $n\ge15$ one has
 \[\zcl(\widetilde G_{n,3})\le2+\zcl(W_n).\]
\end{proposition}
\begin{proof}
Let $t\ge4$ be the integer with the property $2^t-1\le n\le2^{t+1}-2$.
According to \cite[Theorem A]{BasuChakraborty}, there are indecomposable classes in
$H^*(\widetilde G_{n,3})$ outside $W_n$, but at most two of them. Let us denote these classes by $a$ and $b$, where $|a|<|b|$ (in the cases in which there is only one indecomposable class, it is denoted by $a$). Then $|a|=\min\{3n-2^{t+1}-1,2^{t+1}-4\}$, and if $b$ exists,
then $|b|=\max\{3n-2^{t+1}-1,2^{t+1}-4\}$.  Since the dimension of the manifold is $3n-9$, it follows that $\height(a)\le3$ and $\height(b)=1$, and so $\height(z(a))\le3$ and $\height(z(b))=1$ (see (\ref{htz})).

Suppose to the contrary that $\zcl(\widetilde G_{n,3})\ge3+\zcl(W_n)$. Then $\zcl(\widetilde G_{n,3})$ is realized by a monomial of the form $z(a)^pz(b)^qz(\widetilde w_2)^\beta z(\widetilde w_3)^\gamma$, where $p+q\ge3$. By the previous discussion we also have $p+q\le\height(z(a))+\height(z(b))\le4$. Now, $p+q+\beta+\gamma=\zcl(\widetilde G_{n,3})\ge3+\zcl(W_n)$, and we conclude
\begin{equation}\label{ineq1}
\beta+\gamma\ge3+\zcl(W_n)-p-q\ge\zcl(W_n)-1,
\end{equation}
i.e.,
\begin{equation}\label{donjogr}
2(\beta+\gamma)\ge2\zcl(W_n)-2.
\end{equation}

The proof of Lemma 2.3 from \cite{Radovanovic} works equally well for the oriented Grassmannians, and so $p|a|+q|b|+2\beta+3\gamma=|z(a)^pz(b)^qz(\widetilde w_2)^\beta z(\widetilde w_3)^\gamma|\le6(n-3)-1$. On the other hand, since $\beta\le\height(z(\widetilde w_2))$, from (\ref{ineq1}) we get $\gamma\ge\zcl(W_n)-1-\beta\ge\zcl(W_n)-\height(z(\widetilde w_2))-1$. Therefore,
\begin{align*}
2(\beta+\gamma)&\le2(\beta+\gamma)+\gamma-\big(\zcl(W_n)-\height(z(\widetilde w_2))-1\big)\\
               &=2\beta+3\gamma-\zcl(W_n)+\height(z(\widetilde w_2))+1\\
               &\le6(n-3)-1-p|a|-q|b|-\zcl(W_n)+\height(z(\widetilde w_2))+1.
\end{align*}
Note also that $p|a|+q|b|\ge3|a|$ (since $p+q\ge3$ and $|a|<|b|$), and so we have
\begin{equation}\label{gornjogr}
2(\beta+\gamma)\le6n-18-3|a|-\zcl(W_n)+\height(z(\widetilde w_2)).
\end{equation}
According to (\ref{donjogr}) and (\ref{gornjogr}) we will reach a contradiction as soon as we prove
\[6n-18-3|a|-\zcl(W_n)+\height(z(\widetilde w_2))<2\zcl(W_n)-2,\]
i.e.,
\begin{equation}\label{ineq}
6n+\height(z(\widetilde w_2))<3\big(|a|+\zcl(W_n)\big)+16.
\end{equation}

In Table \ref{table:1}
we have listed the values of $|a|$, $\zcl(W_n)$ and $\height(z(\widetilde w_2))$ depending on
$n$. The second column of the table is due to \cite[Theorem A]{BasuChakraborty}, in the third
are our results from Theorem \ref{zcl Wn}, while the fourth follows from Theorem
\ref{heights} and (\ref{htz}).

\begin{table}[h!]
\footnotesize
\centering
\begin{tabular}{||m{2.5cm} m{2.1cm} m{3.5cm} m{1.5cm}||}
 \hline
 $n$ & $|a|$ & $\zcl(W_n)$ & $\height(z(\widetilde w_2))$ \\ [0.5ex]
 \hline\hline
 $2^t-1$ &  &   & \\
 $\cdot$ &  &   &\\
 $\cdot$ & $3n-2^{t+1}-1$ & $2^t+2^{t-1}-4$ &  $2^t-1$\\
 $\cdot$ &  &   & \\
 $2^t+2^{t-2}$ &  &   & \\
 \hline
  $2^t+2^{t-2}+1$ & $3n-2^{t+1}-1$ & $2^t+2^{t-1}-3$ &  $2^t-1$\\
\hline
 $2^t+2^{t-2}+2$ &  &   & \\
 $\cdot$ &  &  &\\
 $\cdot$ & $3n-2^{t+1}-1$ & $2^t+2^{t-1}-2$ &  $2^t-1$\\
 $\cdot$ &  &  & \\
 $2^t+\lfloor2^t/3\rfloor-1$ &  &  & \\
\hline
 $2^t+\lfloor2^t/3\rfloor$ &  &  & \\
 $\cdot$ &  &  &\\
 $\cdot$ & $2^{t+1}-4$ & $2^t+2^{t-1}-2$ & $2^t-1$\\
 $\cdot$ &  &  & \\
 $2^t+2^{t-1}$ &  &  & \\
\hline
  $2^t+2^{t-1}+1$ & $2^{t+1}-4$ & $2^{t+1}+2^{t-3}-3$ & $2^{t+1}-1$\\
\hline
 $2^t+2^{t-1}+2$ &  &  & \\
 $\cdot$ &  &  &\\
 $\cdot$ & $2^{t+1}-4$ & $2^{t+1}+2^{t-3}-2$ & $2^{t+1}-1$\\
 $\cdot$ &  &  & \\
 $2^t+2^{t-1}+2^{t-3}$ &  &  & \\
\hline
 $2^t+2^{t-1}+2^{t-3}+1$ &  &  & \\
 $\cdot$ &  &  &\\
 $\cdot$ & $2^{t+1}-4$ & $2^{t+1}+2^{t-2}-2$ & $2^{t+1}-1$\\
 $\cdot$ &  &  & \\
 $2^t+2^{t-1}+2^{t-2}$ &  &  & \\
\hline
 $2^t+2^{t-1}+2^{t-2}+1$ &  &  & \\
 $\cdot$ &  & $2^{t+1}+2^t-2^{s+1}-2,$  &\\
 $\cdot$ & $2^{t+1}-4$ & where $s\in\{1,\ldots,t-3\}$ & $2^{t+1}-1$\\
 $\cdot$ &  & is such that $2^{t+1}-2^{s+1}+$& \\
 $2^{t+1}-2$ &  & $1\le n\le2^{t+1}-2^s$  & \\
\hline
\end{tabular}
\caption{}
\label{table:1}
\end{table}

It is now routine to check that the inequality (\ref{ineq}) holds in all cases. For instance, if $2^t+2^{t-1}+1\le n\le2^t+2^{t-1}+2^{t-3}$, then
\begin{align*}
6n+\height(z(\widetilde w_2))&\le6(2^t+2^{t-1}+2^{t-3})+2^{t+1}-1=6\cdot13\cdot2^{t-3}+16\cdot2^{t-3}-1\\
                             &=94\cdot2^{t-3}-1,
\end{align*}
while
\begin{align*}
3\big(|a|+\zcl(W_n)\big)+16&\ge3(2^{t+1}-4+2^{t+1}+2^{t-3}-3)+16  \\
                           &=3(2^{t+2}+2^{t-3}-7)+16=99\cdot2^{t-3}-5,
\end{align*}
and $5\cdot2^{t-3}-4>0$ implies (\ref{ineq}).
\end{proof}

So for all integers $n\ge15$ we have
\[1+\zcl(W_n)\le\zcl(\widetilde G_{n,3})\le2+\zcl(W_n).\]
We are not aware of any integer $n$ for which $\zcl(\widetilde G_{n,3})=2+\zcl(W_n)$. Furthermore, we have used the computer software SAGE, and relations obtained in \cite{MatszangoszWendt}, to verify that $\zcl(\widetilde G_{n,3})=1+\zcl(W_n)$ for $6\le n\le100$. Therefore, the following conjecture seems reasonable.

\begin{conjecture}
For all $n\ge6$, $\zcl(\widetilde G_{n,3})=1+\zcl(W_n)$.
\end{conjecture}

In the following proposition we prove the conjecture for approximately $5/12$ of the integers in the range $[2^t-1,2^{t+1}-2]$ (more precisely, in the first sixth and the last quarter of this range).

\begin{proposition}
Let $t\ge4$. If either $2^t-1\le n<2^t+2^{t-1}/3+1$ or $2^t+2^{t-1}+2^{t-2}+1\le n\le2^{t+1}-2$, then
\[\zcl(\widetilde G_{n,3})=1+\zcl(W_n).\]
\end{proposition}
\begin{proof}
Assume to the contrary that $\zcl(\widetilde G_{n,3})=2+\zcl(W_n)$. This means that $\zcl(\widetilde G_{n,3})$ is reached by a monomial of the form $z(a)^pz(b)^qz(\widetilde w_2)^\beta z(\widetilde w_3)^\gamma$ with $p+q\ge2$, where $a$ and $b$ (if $b$ exists) are the two indecomposable classes outside $W_n$, as in the proof of Proposition \ref{c2}. In that proof we showed that $p+q$ cannot exceed $2$, so $p+q=2$. Since $\height(z(b))=1$ (if $b$ exists), $\zcl(\widetilde G_{n,3})$ is reached by either $z(a)^2z(\widetilde w_2)^\beta z(\widetilde w_3)^\gamma$ or $z(a)z(b)z(\widetilde w_2)^\beta z(\widetilde w_3)^\gamma$, where $\beta+\gamma=\zcl(W_n)$.

\medskip

Suppose that $\zcl(\widetilde G_{n,3})$ is reached by $z(a)^2z(\widetilde w_2)^\beta z(\widetilde w_3)^\gamma$. Note that
\[z(a)^2z(\widetilde w_2)^\beta z(\widetilde w_3)^\gamma=z(a^2)z(\widetilde w_2)^\beta z(\widetilde w_3)^\gamma=(a^2\otimes1+1\otimes a^2)z(\widetilde w_2)^\beta z(\widetilde w_3)^\gamma\]
is a sum of simple tensors of the form $a^2u\otimes v$ and $u\otimes a^2v$, where $u,v\in W_n$ are such that $|u|+|v|=2\beta+3\gamma$. Since $z(a)^2z(\widetilde w_2)^\beta z(\widetilde w_3)^\gamma\neq0$, at least one of these simple tensors is nonzero. If $a^2u\otimes v$ is a nonzero simple tensor, then $a^2u\neq0$ in $H^*(\widetilde G_{n,3})$ and $v\neq0$ in $W_n$ (and similarly if $u\otimes a^2v\neq0$). This means that
\begin{equation}\label{a^2}
2|a|+|u|\le3n-9 \quad \mbox{ and } \quad |v|\le3n-9-|a|.
 \end{equation}
Namely, the dimension of the manifold $\widetilde G_{n,3}$ is $3n-9$, and since $v\in W_n$ is nonzero, Poincar\'e duality gives us a class $x\in H^{3n-9-|v|}(\widetilde G_{n,3})$ with the property $vx\neq0$. But if $|v|$ were greater than $3n-9-|a|$, then we would have $|x|<|a|$, which would mean that $x\in W_n$ (since $a$ is the class of smallest degree in $H^*(\widetilde G_{n,3})$ which is not in $W_n$), and so $vx\in W_n$. This would contradict (\ref{topdim}).

Similarly as in the proof of Proposition \ref{c2}, we also have
\[\gamma=\beta+\gamma-\beta=\zcl(W_n)-\beta\ge\zcl(W_n)-\height(z(\widetilde w_2)).\]
Using this and summing up the inequalities form (\ref{a^2}) we get
\begin{align*}
2\zcl(W_n)&\le2(\beta+\gamma)+\gamma-\big(\zcl(W_n)-\height(z(\widetilde w_2))\big)\\
          &=2\beta+3\gamma-\zcl(W_n)+\height(z(\widetilde w_2))=|u|+|v|-\zcl(W_n)+\height(z(\widetilde w_2))\\
          &\le6n-18-3|a|-\zcl(W_n)+\height(z(\widetilde w_2)),
\end{align*}
that is
\[6n+\height(z(\widetilde w_2))\ge3\big(|a|+\zcl(W_n)\big)+18.\]
However, this contradicts (\ref{ineq}). So $\zcl(\widetilde G_{n,3})$ is not realized by $z(a)^2z(\widetilde w_2)^\beta z(\widetilde w_3)^\gamma$.

\medskip

This proves the proposition for $n\in\{2^t-1,2^t,2^{t+1}-3,2^{t+1}-2\}$, since in these cases $a$ is the only (up to addition of an element from $W_n$) indecomposable class outside $W_n$ (i.e., $b$ does not exist in these cases).

\medskip

Now suppose that $n\notin\{2^t-1,2^t,2^{t+1}-3,2^{t+1}-2\}$ and that $\zcl(\widetilde G_{n,3})$ is reached by $z(a)z(b)z(\widetilde w_2)^\beta z(\widetilde w_3)^\gamma$, $\beta+\gamma=\zcl(W_n)$. Let us note that, as in the proof of Proposition \ref{c2}, $|a|+|b|+2\beta+3\gamma=|z(a)z(b)z(\widetilde w_2)^\beta z(\widetilde w_3)^\gamma|\le6n-19$. On the other hand, according to \cite[Theorem A]{BasuChakraborty}, $|a|+|b|=3n-5$, and similarly as before we have $\gamma\ge\zcl(W_n)-\height(z(\widetilde w_2))$, so
\begin{align*}
3n-5+2\zcl(W_n)&\le|a|+|b|+2(\beta+\gamma)+\gamma-\big(\zcl(W_n)-\height(z(\widetilde w_2))\big)\\
          &=|a|+|b|+2\beta+3\gamma-\zcl(W_n)+\height(z(\widetilde w_2))\\
          &\le6n-19-\zcl(W_n)+\height(z(\widetilde w_2)).
\end{align*}
We conclude that $3n+\height(z(\widetilde w_2))\ge3\zcl(W_n)+14$.

However, if $2^t+1\le n<2^t+2^{t-1}/3+1$, then by looking at Table \ref{table:1} we see that
\[3n+\height(z(\widetilde w_2))<3(2^t+2^{t-1}/3+1)+2^t-1=2^{t+2}+2^{t-1}+2,\]
while
\[3\zcl(W_n)+14\ge3(2^t+2^{t-1}-4)+14=2^{t+2}+2^{t-1}+2.\]
Similarly, if $2^t+2^{t-1}+2^{t-2}+1\le n\le2^{t+1}-4$, then
\[3n+\height(z(\widetilde w_2))\le3(2^{t+1}-4)+2^{t+1}-1=2^{t+3}-13,\]
and
\[3\zcl(W_n)+14\ge3(2^{t+1}+2^t-2^{t-2}-2)+14=2^{t+3}+2^{t-2}+8.\]
This contradiction concludes the proof.
\end{proof}

We have seen that there exist numbers $n$ with $\zcl(\widetilde G_{n,3})=1+\zcl(W_n)$, so in terms of $\zcl(W_n)$, $1+\zcl(W_n)$ is the best general lower bound for $\zcl(\widetilde G_{n,3})$ that one could get. Now, since $1+\zcl(\widetilde G_{n,3})$ is a lower bound for $\mathrm{TC}(\widetilde G_{n,3})$, Theorem \ref{zcl Wn} provides lower bounds for $\mathrm{TC}(\widetilde G_{n,3})$ as given in Table \ref{table:3}.

\begin{table}[h!]
\footnotesize
\centering
\begin{tabular}{||m{2.4cm} m{3.3cm} m{2.9cm} m{2.3cm}||}
 \hline
 $n$ & $\zcl(W_n)$ & $\zcl(\widetilde G_{n,3})$ & $\mathrm{TC}(\widetilde G_{n,3})$ \\ [0.5ex]
 \hline\hline
 $2^t-1$ &  &   & \\
 $\cdot$ &  &   &\\
 $\cdot$ & $2^t+2^{t-1}-4$ &  $2^t+2^{t-1}-3$ &  $\ge2^t+2^{t-1}-2$\\
 $\cdot$ &  & & \\
 $2^t+\lfloor2^{t-1}/3\rfloor+1$ &  & & \\
 \hline
 $2^t+\lfloor2^{t-1}/3\rfloor+2$ &  &   & \\
 $\cdot$ &  &   &\\
 $\cdot$ & $2^t+2^{t-1}-4$ &  $\ge2^t+2^{t-1}-3$ &  $\ge2^t+2^{t-1}-2$\\
 $\cdot$ &  & & \\
 $2^t+2^{t-2}$ &  & & \\
 \hline
  $2^t+2^{t-2}+1$ & $2^t+2^{t-1}-3$ &  $\ge2^t+2^{t-1}-2$ &  $\ge2^t+2^{t-1}-1$\\
\hline
 $2^t+2^{t-2}+2$  & &  & \\
 $\cdot$ & & &\\
 $\cdot$ &  $2^t+2^{t-1}-2$ &  $\ge2^t+2^{t-1}-1$ &  $\ge2^t+2^{t-1}$\\
 $\cdot$ & &  & \\
 $2^t+2^{t-1}$ & & & \\
\hline
  $2^t+2^{t-1}+1$ & $2^{t+1}+2^{t-3}-3$ & $\ge2^{t+1}+2^{t-3}-2$ & $\ge2^{t+1}+2^{t-3}-1$\\
\hline
 $2^t+2^{t-1}+2$ & & & \\
 $\cdot$ & &  &\\
 $\cdot$ & $2^{t+1}+2^{t-3}-2$ & $\ge2^{t+1}+2^{t-3}-1$ & $\ge2^{t+1}+2^{t-3}$\\
 $\cdot$ & & &\\
 $2^t+2^{t-1}+2^{t-3}$ & & & \\
\hline
 $2^t+2^{t-1}+2^{t-3}+1$ & & & \\
 $\cdot$ & & &\\
 $\cdot$ & $2^{t+1}+2^{t-2}-2$ & $\ge2^{t+1}+2^{t-2}-1$ & $\ge2^{t+1}+2^{t-2}$\\
 $\cdot$ &  & & \\
 $2^t+2^{t-1}+2^{t-2}$ & &  & \\
\hline
 $2^t+2^{t-1}+2^{t-2}+1$ & & & \\
 $\cdot$  & $2^{t+1}+2^t-2^{s+1}-2,$ &  &\\
 $\cdot$  & where $s\in\{1,\ldots,t-3\}$ & $2^{t+1}+2^t-2^{s+1}-1$ & $\ge2^{t+1}+2^t-2^{s+1}$\\
 $\cdot$ & is such that $2^{t+1}-2^{s+1}+$ & & \\
 $2^{t+1}-2$ & $1\le n\le2^{t+1}-2^s$ & & \\
\hline
\end{tabular}
\caption{}
\label{table:3}
\end{table}

\bibliographystyle{amsplain}

\end{document}